\documentclass[10 pt]{article}
\usepackage{amssymb,amsthm}
\font\goth=eufm10
\font\bigmath=cmsy10 scaled \magstep 2

\newcommand{\bigtimes}{\hbox{\bigmath \char'2}}
\newcommand{\gc}{\hbox{\goth c}}

\newcommand{\cl}{c\ell}

\newcommand{\mod}{\hbox{\rm mod}}
\newcommand{\emp}{\emptyset}
\newcommand{\bea}{\mathbb A}
\newcommand{\ben}{\mathbb N}

\newcommand{\beq}{\mathbb Q}
\newcommand{\bez}{\mathbb Z}

\newcommand{\A}{\mathbb A}
\newcommand{\nhat}[1]{\{1,2,\ldots,#1\}}
\newcommand{\ohat}[1]{\{0,1,\ldots,#1\}}
\newcommand{\pf}{{\mathcal P}_f}

\newcommand{\Jo}{J_{T_1}(A_1)}
\newcommand{\Jt}{J_{T_2}(A_2)}
\newcommand{\Jot}{J_{T_1\times T_2}(A_1\times A_2)}
\newcommand{\varwords}[2]{\widetilde S\hskip -1 pt\Big(\hskip -4 pt\hbox{$\begin{array}{c}#1\\ #2
\end{array}$}\hskip -4 pt\Big)}
\newcommand{\varwordspl}[2]{S\hskip -1 pt\Big(\hskip -4 pt\hbox{$\begin{array}{c}#1\\ #2
\end{array}$}\hskip -4 pt\Big)}
\newcommand{\varwordsj}[3]{\widetilde S^{(#1)}\hskip -2 pt\Big(\hskip -4 pt\hbox{$\begin{array}{c}#2\\ #3
\end{array}$}\hskip -4 pt\Big)}
\newtheorem{theorem}{Theorem}[section]
\newtheorem{corollary}[theorem]{Corollary}
\newtheorem{lemma}[theorem]{Lemma}
\newtheorem{question}[theorem]{Question}
\theoremstyle{definition}
\newtheorem{definition}[theorem]{Definition}

\title{Combining extensions of the Hales-Jewett\\
       Theorem with Ramsey Theory\\
in other structures}

\date{}
\author{Neil Hindman
        \footnote{Department of Mathematics,
                 Howard University,
                  Washington, DC 20059, USA.\hfill\break
                  {\tt nhindman@aol.com}}
\and
Dona Strauss
        \footnote{Department of Pure Mathematics,
        University of Leeds,
        Leeds LS2 9J2, UK. \hfill\break
        {\tt d.strauss@hull.ac.uk}}
\and
Luca Q. Zamboni
        \footnote{Universit\'e de Lyon,
Universit\'e Lyon 1, CNRS UMR 5208,
Institut Camille Jordan,
43 boulevard du 11 novembre 1918,
F69622 Villeurbanne Cedex, France\hfill\break
        {\tt zamboni@math.univ-lyon1.fr}
\hfill\break
This work was supported in part by the LABEX MILYON (ANR-10-LABX-0070) of 
Universit\'e de Lyon, within the program ``Investissements d'Avenir'' 
(ANR-11-IDEX- 0007) operated by the French National Research Agency (ANR).
}
}

\begin{document}

\maketitle

\begin{abstract} The Hales-Jewett Theorem states that given any finite nonempty set $\A$ and any finite coloring of the 
free semigroup $S$ over the alphabet $\A$ there is a {\it variable word\/} over $\A$ all of whose instances
are the same color. This theorem has some extensions 
involving several distinct variables occurring in the variable word. We show that,
when combined with a sufficiently well behaved homomorphism, 
the relevant variable word simultaneously satisfies a
Ramsey-Theoretic conclusion in the other structure. As an example we show that
if $\tau$ is the homomorphism from the set of variable words into the natural numbers which 
associates to each variable word $w$ the number of occurrences of the
variable in $w$, then given any finite coloring of $S$ and any infinite sequence of natural numbers, 
there is a variable word $w$ whose instances are monochromatic and $\tau(w)$ is a sum of distinct members
of the given sequence.

Our methods rely on the algebraic structure of the Stone-\v Cech compactification
of $S$ and the other semigroups that we consider.  We show for example
that if $\tau$ is as in the paragraph above, there is a compact
subsemigroup $P$ of $\beta\ben$ which contains all of the idempotents
of $\beta\ben$ such that, given any $p\in P$, any
$A\in p$, and any finite coloring of $S$, there is a variable word
$w$ whose instances are monochromatic and $\tau(w)\in A$.

We end with a new short algebraic proof of an infinitary extension of the 
Graham-Rothschild Parameter Sets Theorem. 
\end{abstract}

\section{Introduction}

Given a nonempty set $A$ (or {\it alphabet}) we let $A^+$ be the set of all finite
words $w=a_1a_2\cdots a_n$ with $n\geq 1$ and $a_i\in A$. The quantity $n$ is called the {\it length\/} of $w$ and 
denoted $|w|$. The set $A^+$ is naturally a semigroup under the operation of concatenation of words, known as
the {\it free semigroup\/} over $A$.  For each $u\in A^+$ and $a\in A$, we let $|u|_a$ be the number of occurrences of $a$ in $u$. 
For $u,w \in A^+$ we say $u$ and $w$ are Abelian equivalent, and write 
$u \sim_{Ab}w$, whenever $|u|_a=|w|_a$ for all $a\in A$. As is customary, we will identify the elements of
$A$ with the length one words over $A$.

Throughout this paper we will let $\A$ be a nonempty set, let $S_0=\A^+$ be the free semigroup over $\A$,  
and let $v$ (a {\it variable\/}) be a letter not belonging to $\A$. By a {\it variable word\/} over $\A$ we mean a  word
$w$ over $\A\cup\{v\}$ with $|w|_v\geq 1$. We let $S_1$ be the set
of variable words over $\A$. If $w\in S_1$ and $a\in \A$, then 
$w(a)\in S_0$ is the result of replacing each occurrence of $v$ by $a$.
For example if $\A=\{a,b,c\}$ and $w=avbvva$, then $w(a)=aabaaa$ while
$w(c)=acbcca$. A {\it finite coloring\/} of a set $X$ is a function
from $X$ to a finite set.  A subset $A$ of $X$ is {\it monochromatic\/}
if the function is constant on $A$.

\begin{theorem}[A. Hales and R. Jewett]\label{HJth} Assume that $\A$ is finite.
For each finite coloring of $S_0$ there exists a variable word
$w$ such that $\{w(a):a\in \A\}$ is monochromatic.\end{theorem}

\begin{proof} \cite[Theorem 1]{HJ}. \end{proof}

Some extensions of the Hales-Jewett Theorem, including for example Theorem~\ref{ndimHJ} or the Graham-Rothschild Parameter Sets Theorem \cite{GR},  involve the notion of {\it $n$-variable words\/}.

\begin{definition}\label{defnvar} Let $n\in\ben$ and $v_1,v_2,\ldots ,v_n$ be distinct variables which are not members of $\A$.
\begin{itemize}\item[(a)] An {\it $n$-variable word\/} over $\A$ is a word $w$ over
$\A\cup\{v_1,v_2,\ldots,v_n\}$ such that $|w|_{v_i}\geq 1$ for each  $i\in\nhat{n}$. 
\item[(b)] If $w$ is an $n$-variable word over $\A$ and $\vec x=(x_1,x_2,\ldots,x_n)$,
then $w(\vec x)$ is the result of replacing each occurrence of $v_i$ in $w$ by $x_i$ for
each $i\in\nhat{n}$. 
\item[(c)] If $w$ is an $n$-variable word over $\A$ and $u=l_1l_2\cdots l_n$ is a length $n$
word, then $w(u)$ is the result of replacing each occurrence of $v_i$ in $w$ by $l_i$ for
each $i\in\nhat{n}$. 
\item[(d)] A {\it strong $n$-variable word\/} is an $n$-variable word such that
for each $i\in\nhat{n-1}$, the first occurrence of $v_i$ precedes the first
occurrence of $v_{i+1}$.
\item[(e)] $S_n$ is the set of $n$-variable words over $\A$ and $\widetilde S_n$ is the set of
strong $n$-variable words over $\A$.
\item[(f)] $\widetilde S_0=S_0$.
\item[(g)] If $m\in\omega=\ben\cup\{0\}$ and $m<n$, then $\varwords nm$ is the set
of $u\in \widetilde S_m$ such that $|u|=n$.
\item[(h)] If $m\in\omega$ and $m<n$, then $\varwordspl nm$ is the set
of $u\in S_m$ such that $|u|=n$.

\end{itemize}
\end{definition}

The notation above does not reflect the dependence on the alphabet $\A$.

We note that if $m,n \in \omega$ and $m<n$, then for each $w\in \widetilde S_n$ and each $u\in\varwords nm$, the word $w(u)$ belongs to $\widetilde S_m$.

The following is a first simple example of a  multivariable extension of the Hales-Jewett Theorem: 

\begin{theorem}\label{ndimHJ} Assume that $\A$ is finite. Let
$S_0$ be finitely colored and let $n\in\ben$.  There exists $w\in S_n$ 
such that $\{w(\vec x):\vec x\in \A^n\}$ is monochromatic.\end{theorem}

Theorem~\ref{ndimHJ} follows immediately from Theorem \ref{HJth} applied to
the alphabet $\A^n$, replacing each occurrence of $v_1$ in the
variable word over $\A^n$ by $v_1v_2\cdots v_n$.  It is also a consequence of  Theorem \ref{matrix}, which constitutes one of the main results of this paper.
(See the paragraph immediately following Theorem \ref {charIPR}.)
Theorem~\ref{ndimHJ} also follows directly from Theorem~\ref{blass} later in this section 
which we regard as an algebraic extension of Theorem~\ref{ndimHJ}. 

It is natural to ask the following question. {\it Assume that $\A$ is finite. Let
$S_\infty$ be the set of infinite words over $\A\cup\{v_i:i\in\ben\}$ in which
each $v_i$ occurs and assume that $\A^{\ben}$ is finitely colored. Must there
exist $w\in S_\infty$ such that $\{w(\vec x):\vec x\in\A^{\ben}\}$ is monochromatic,
where $w(\vec x)$ has the obvious meaning?\/} As long as $|\A|\geq 2$, the answer is
easily seen to be ``no", using a standard diagonalization argument:  One has that
$ |\A^\ben|=|S_\infty|=\gc$, so one may inductively color two elements of
$\A^\ben$ for each $w \in S_\infty$ so that there exist $\vec x$ and $\vec y$ in 
$\A^\ben$ with the color of $w(\vec x)$ and $w(\vec y)$ different. 
(When one gets to $w$, fewer than $\gc$ things have been colored and
there are $\gc$ distinct values of $w(\vec x)$ possible.)

The following simplified version of the Graham-Rothschild Parameter Sets Theorem constitutes yet another fundamental multivariable extension
of the Hales-Jewett Theorem. It was shown in \cite[Theorem 5.1]{CHS} that the full version as stated in 
\cite{GR} can be easily derived from the version stated here.

\begin{theorem}[R. Graham and B. Rothschild]\label{GRth} Assume that $\A$ is finite. 
Let $m,n\in\omega$ with $m<n$
and let $\widetilde S_m$ be finitely colored.  There exists $w\in \widetilde S_n$ such that
$\{w(u):u\in\varwords nm\}$ is monochromatic.\end{theorem}

\begin{proof}\cite[Section 7]{GR}.\end{proof}

After identifying the elements of $\A$ with the length $1$ words over $\A$, one
sees that Theorem \ref{HJth} is exactly the $m=0$ and $n=1$ case of Theorem \ref{GRth}.
Notice also that  Theorem \ref{ndimHJ} is actually equivalent to Theorem \ref{GRth} in the special case of $m=0$.
In fact if $w\in S_n$, $\sigma$ is a 
permutation of $\nhat{n}$ and $u$ is the result of replacing
each $v_i$ in $w$  by $v_{\sigma(i)}$ for each $i\in\nhat{n}$, then
 $\{u(\vec x):\vec x\in \A^n\}=\{w(\vec x):\vec x\in \A^n\}$.
In this paper we shall be mostly concerned with cases of Theorem \ref{GRth} with $m=0$ and
arbitrary $n\in\ben$. (We are not concerned with $m>0$ because the natural
versions of our main theorems are not valid for $m>0$. We shall discuss this
point at the end of Section \ref{seccombine}.)  Accordingly, from this point on until Section \ref{infext} we will not be concerned with
the order of occurrence of the variables.

In contrast to Theorem~\ref{ndimHJ}, the Graham-Rothschild Parameter 
Sets Theorem does not appear to be deducible directly from the Hales-Jewett Theorem; at least we 
know of no such proof. In Section \ref{infext} we present a new purely algebraic proof 
of an infinitary extension of Theorem~\ref{GRth}.

Our main results in this paper deal with obtaining
$n$-variable words satisfying the Hales-Jewett Theorem and
simultaneously relating to Ramsey-Theoretic results
in some relevant semigroup. The paper is organized as follows:

In Section \ref{seccombine} we present our main theorems
relating $S_n$ with other structures.  
In Section \ref{sechomo} we determine precisely which
homomorphisms from $S_n$ to $(\ben,+)$ satisfy 
the hypotheses of our main theorem of Section \ref{seccombine},
namely Theorem \ref{matrix}.

As consequences of the results of Section \ref{seccombine}
we establish that for $k\in\ben$, the set 
of points $(p_1,p_2,\ldots,p_k)\in (\beta\ben)^k$ with
the property that whenever $B_i\in p_i$ for $i\in\nhat{k}$, the $k$-tuple
$(B_1,B_2,\ldots,B_k)$ satisfies the conclusions of
one of those theorems, is a compact subsemigroup of $(\beta\ben)^k$
containing the idempotents of $(\beta\ben)^k$ (or the minimal idempotents, depending on
the theorem). The details of these results 
will be presented in Section \ref{seccptsemi}.

In Section \ref{seccptideals} we restrict our attention to
versions of the Hales-Jewett Theorem. Letting
$R_n=\{p\in\beta S_0:(\forall B\in p)(\exists w\in S_n)(\{w(\vec x):\vec x\in \A^n\}\subseteq B)\}$.
We show that each $R_n$ is a compact ideal of $\beta S_0$, that $R_{n+1}\subsetneq R_n$ for
each $n\in\ben$, and that $\cl K(\beta S_0)\subsetneq\bigcap_{n=1}^\infty R_n$.

In Section \ref{infext} we present a new fully algebraic proof of an infinitary extension of the
Graham-Rothschild Parameter Sets Theorem. This new proof is a significant
simplification of the original.

The statements and proofs of the results in this paper use strongly the algebraic structure
of the Stone-\v Cech compactification of a discrete 
semigroup. We now present a brief description 
of this structure.  For more details or for  any unfamiliar facts encountered in
this paper, we refer the reader to  \cite[Part I]{HS}. All 
topological spaces considered herein are assumed to be Hausdorff.

Let $S$ be a semigroup. For each $s\in S$, $\rho_s:S\to S$ and $\lambda_s:S\to S$ are defined by
$\rho_s(x)=xs$ and $\lambda_s(x)=sx$.  If $S$ is also a topological space, $S$ is said to 
be {\it right topological\/} if the map 
$\rho_s$ is continuous for every $s\in S$. In this case, the set of elements $s\in S$ for 
which $\lambda_s$ is continuous, is called the {\it topological center\/} of $S$.

The assumption that $S$ is compact and right topological
has powerful algebraic implications. $S$ has a smallest two sided ideal $K(S)$ which is the
union of all of the minimal right ideals, as well as
the union of all of the minimal left ideals. The intersection
of any minimal left ideal and any minimal right ideal is
a group. In particular, $S$ has idempotents.   Any left ideal of $S$ contains a 
minimal left ideal of $S$, and any right ideal of $S$ contains a minimal right ideal of $S$.  So the intersection of any left ideal of $S$
and any right ideal of $S$ contains an idempotent. An idempotent in $S$ is said 
to be {\it minimal\/} if it is in $K(S)$. This is equivalent to being minimal in 
the ordering of idempotents defined by  $p\leq q$ if
$pq=qp=p$. If $q$ is any idempotent in $S$, there is a minimal idempotent $p\in S$ for which $p\leq q$.

Given a discrete semigroup $(T,\cdot)$, let $\beta T=\{p:
p$ is an ultrafilter on $T\}$. We identify the
principal ultrafilter $e(x)=\{A\subseteq T:x\in A\}$ 
with the point $x\in T$ and thereby pretend that
$T\subseteq \beta T$. A base for the topology
of $\beta T$ consists of the clopen sets $\overline A$ for
all $A\subseteq T$, where $\overline A=\{p\in\beta T:A\in p\}$. The operation $\cdot$ on $T$
extends to an operation on $\beta T$, also denoted
by $\cdot$ making $(\beta T,\cdot)$ a right topological
semigroup with $T$ contained in its topological center. So, given $p,q\in \beta T$,
$p\cdot q=\displaystyle\lim_{s\to p}\lim_{t\to q}s\cdot t$, where $s$ and $t$ denote elements of $T$.
  If $A\subseteq T$,
$A\in p\cdot q$ if and only if $\{x\in T:x^{-1}A\in q\}\in p$,
where $x^{-1}A=\{y\in T:x\cdot y\in A\}$. If $(T,+)$ is a commutative discrete semigroup, we will use
+ for the semigroup operation on $\beta T$, even though $\beta T$ is likely to be far from commutative. 
In this case, we have that $A\in p+ q$ if and only if $\{x\in T:-x+A\in q\}\in p$,
where $-x+A=\{y\in T:x+ y\in A\}$.

 A set $D\subseteq T$
is  {\it piecewise syndetic\/} if and only if $D\in p$ for some $p\in K(\beta T)$ and 
is {\it central\/} if and only $D\in p$ for some  idempotent $p\in K(\beta T)$.
We will also need the following equivalent characterization 
of piecewise syndetic sets: $D$ is piecewise syndetic if and only if there exists a finite subset $G$ of $T$ with the 
property that for every finite subset $F$ of $T$ there exists $x\in T$ such that $Fx\subseteq \bigcup_{t\in G}t^{-1}D$.
(See \cite[Theorem 4.40]{HS}.)
Given a 
sequence $\langle x_n\rangle_{n=1}^\infty$ and $m\in\ben$, we set
$FP(\langle x_n\rangle_{n=m}^\infty)=
\{\prod_{t\in F}x_t:F\in\pf(\ben)\hbox{ and }\min F\geq m\}$,
where $\pf(\ben)$ is the set of finite nonempty subsets of
$\ben$ and the products are computed in increasing order of indices. 
Then $\bigcap_{m=1}^\infty\overline{FP(\langle x_n\rangle_{n=m}^\infty)}$
is a compact semigroup so there is an idempotent 
$p$ with $FP(\langle x_n\rangle_{n=m}^\infty)\in p$ for every $m$. (See \cite[Lemma 5.11]{HS}.)
If the operation is denoted by $+$, we write
$FS(\langle x_n\rangle_{n=m}^\infty)=
\{\sum_{t\in F}x_t:F\in\pf(\ben)\hbox{ and }\min F\geq m\}$.
Given an idempotent $p$ and $B\in p$ let $B^\star(p)=\{x\in B:x^{-1}B\in p\}$.
Then  $B^\star(p)\in p$ and for each $x\in B^\star(p)$, one has that 
$x^{-1}B^\star(p)\in p$. (See \cite[Lemma 4.14]{HS}). If there is no risk of confusion, we will sometimes write $B^\star$
for $B^\star(p)$.

If $\gamma$ is a function from the discrete semigroup $T$ to a compact space $C$,
then $\gamma$ has a continuous extension from $\beta T$ to $C$, which we will
also denote by $\gamma$. If $\gamma:T\to W$, where $W$ is discrete, we will
view the continuous extension as taking $\beta T$ to $\beta W$, unless we state otherwise.
If $\gamma:T\to C$ is a homomorphism from $T$ into a compact right topological semigroup $C$, with $\gamma[T]$
contained in the topological center of $C$, then the continuous extension
$\gamma:\beta T\to C$ is a homomorphism by  \cite[Corollary 4.22]{HS}.

We end this section with a few simple  illustrations of how the algebraic structure  described above may be applied to derive simple algebraic proofs
of some of the results discussed earlier including for instance the Hales-Jewett Theorem. 
We begin with the following theorem whose proof is based on an argument due to 
Andreas Blass which first appeared in \cite{BBH}. 

\begin{theorem}\label{blass} Let $T$ be a semigroup and let $S$ be a subsemigroup of $T$. 
Let $F$ be a  nonempty set of homomorphisms mapping $T$ to $S$ which are equal to the identity on $S$. 
\begin{itemize}
\item[(1)] Let $p$ be a minimal idempotent in $\beta S$. Let $q$ be an idempotent in $\beta T$
for which $q\leq p$. Then $\nu(q)=p$ for every $\nu\in F$. 
\item[(2)] For any finite subset $F_0$ of $F$ and any central subset $D$ of $S$,  there is a 
central subset $Q$ of $T$ such that, for every $t\in Q$, $\{\nu(t):\nu\in F_0\}\subseteq D$.
\item[(3)] For any finite subset $F_0$ of $F$ and any finite coloring of $S$,  there is a 
central subset $Q$ of $T$ such that, for every $t\in Q$, $\{\nu(t):\nu\in F_0\}$
is monochromatic. \end{itemize}
\end{theorem}

\begin{proof} (1) For each $\nu\in F$,
$\nu(q)\leq \nu(p)=p$ and so, since $\nu(q)\in \beta S$, $\nu(q)=p$. 

(2) Pick a minimal idempotent $p\in\beta S$ such that $D\in p$. 
By \cite[Theorem 1.60]{HS},
pick a minimal idempotent $q\in\beta T$ such that $q\leq p$. Then $\nu(q)=p$ for every $\nu\in F_0$. 
Hence, if $Q=\bigcap_{\nu\in F_0}\nu^{-1}[D]$,  then $Q\in q$. 

(3) Pick a minimal idempotent $p\in\beta S$ and let $D$ be a monochromatic member of $p$.
\end{proof}

We note that the above theorem provides an algebraic proof of Theorem~\ref{ndimHJ} and hence of the Hales-Jewett Theorem. 
In fact, put $S=S_0$, $T=S_0\cup S_n$
and $F=\{h_{\vec x}: \vec x \in \A^n\}$, where $h_{\vec x}(w)=\left\{\begin{array}{cl}w(\vec x)&\hbox{if }w\in S_n\\
w&\hbox{if }w\in S_0\,.\end{array}\right.$ Then by Theorem~\ref{blass} we deduce that for any finite coloring of $S_0$ there exists
a central subset $Q$ of $T$ such that for every $w\in Q$, $\{w(\vec x):\vec x \in \A^n\}$ is monochromatic. 
Pick $q\in K(\beta T)$ with $Q\in q$.
Then since $S_n$ is an ideal of $T$ it follows that $S_n\in q$. So for any $w\in S_n\cap Q$ 
we have $\{w(\vec x):\vec x\in \A^n\}$ is monochromatic.

We conclude this section with two additional simple corollaries of Theorem~\ref{blass} that
will not be needed in the rest of the paper.

\begin{corollary}\label{corps} Let $T$ be a semigroup and let $S$ be a subsemigroup of  $T$.
Let $F$ be a finite nonempty set of homomorphisms mapping $T$ to $S$ which are equal to the identity on $S$. Let $D$
be a piecewise syndetic subset of $S$. Then $\bigcap_{\nu \in F}\nu^{-1}[D]$ is a piecewise syndetic subset
of $T$. \end{corollary}

\begin{proof} By \cite[Theorem 4.43]{HS}, we may pick $s\in S$ for which $s^{-1}D$ is a 
central subset of $S$. We can  choose a minimal idempotent $p$ in $\beta S$ for 
which $s^{-1}D\in p$, and we can then choose a
minimal idempotent $q$ in $\beta T$ for which $q\leq p$, by \cite[Theorem 1.60]{HS}. 
By Theorem \ref{blass}(1), $\nu(q)=p$ for every $\nu\in F$. Hence, if 
$Q=\bigcap_{\nu\in F}\nu^{-1}[s^{-1}D]$, then $Q\in q$.
Now $sQ$ is a piecewise syndetic subset of $T$, because $sQ\in sq$ and $sq\in K(\beta T)$. 
We claim that $sQ\subseteq \bigcap_{\nu \in F}\nu^{-1}[D]$. In fact, let 
$x\in sQ$, pick $t\in Q$ such that $x=st$, and let
$\nu\in F$. Then $\nu(x)=\nu(st)=s\nu(t)\in s(s^{-1}D)\subseteq D$.
\end{proof}

 In the following corollary, we use the abbreviated notation  $P^\star$ for $P^\star(p)$ for $P$ belonging to an idempotent $p$.

\begin{corollary}\label{corblass}  Let $T$ be a semigroup and let $S$ be a subsemigroup of  $T$.
Let $F$ be a finite nonempty set  of homomorphisms from $T$ onto $S$ which are equal to the identity on $S$.  
Let $p$ be a minimal idempotent in $\beta S$ and let $P\in p$.
Let $q$ be a minimal idempotent of $\beta T$ for which $q\leq p$ and let 
$Q=\bigcap_{\nu\in F}\nu^{-1}[P^\star]$. Then $Q\in q $. There is an 
infinite sequence $\langle w_n\rangle_{n=1}^{\infty}$ of elements of $Q$ such that
for each $H\in\pf(\ben)$ and each $\varphi:H\to F$,
$\prod_{t\in H}\varphi(t)(w_t)\in P^\star$, where the product is computed in
increasing order of indices.\end{corollary}

\begin{proof} Choose $w_1\in Q$. Let $m\in\ben$ and assume we have
chosen $\langle w_t\rangle_{t=1}^m$ in $Q$ such that whenever
$\emp\neq H\subseteq\nhat{m}$ and $\varphi:H\to F$,
$\prod_{t\in H}\varphi(t)(w_t)\in P^\star$. Note that this hypothesis is
satisfied for $m=1$. Let $$\textstyle E=\{\prod_{t\in H}\varphi(t)(w_t):
\emp\neq H\subseteq\nhat{m}\hbox{ and }\varphi:H\to F\}\,.$$ Then
$E\subseteq P^\star$. Let $R=P^\star\cap\bigcap_{y\in E}y^{-1}P^\star$.
Then $R\in p$ so $\bigcap_{\nu\in F}\nu^{-1}[R]\in q$.
Pick $w_{m+1}\in \bigcap_{\nu\in F}\nu^{-1}[R]$ and note that $w_{m+1}\in Q$.

To verify the hypothesis let $\emp\neq H\subseteq\nhat{m+1}$ and let
$\varphi:H\to F$. If $m+1\notin H$, the conclusion holds by assumption,
so assume that $m+1\in H$. If $H=\{m+1\}$, then $w_{m+1}\in\varphi(m+1)^{-1}[P^\star]$,
so assume that $\{m+1\}\subsetneq H$ and let $G=H\setminus\{m+1\}$. Let
$y=\prod_{t\in G}\varphi(t)(w_t)$. Then
$w_{m+1}\in\varphi(m+1)^{-1}[y^{-1}P^\star]$ so
$\prod_{t\in H}\varphi(t)(w_t)=y\varphi(m+1)(w_{m+1})\in P^\star$.\end{proof}

\section{Combining structures}\label{seccombine}

Throughout this section, and up until Section \ref{infext}, 
 $\A$ is a fixed non-empty finite alphabet. 
Most of the results in this paper involve families of well behaved homomorphisms between certain semigroups:

\begin{definition}\label{Spreserving} Let $n\in \ben$ and let 
$\nu:S_n\to S_0$ be a homomorphism.
We shall say that $\nu$ is {\it $S_0$-preserving\/} if 
$\nu(uw)=u\nu(w)$ and $\nu(wu)=\nu(w)u$ for every $u\in S_0$ and every $w\in S_n$. \end{definition}

Note that if $\vec x\in\bea^n$, then the function $h_{\vec x}:S_n\to S_0$ defined
by $h_{\vec x}(w)=w(\vec x)$ is an $S_0$-preserving homomorphism. Also, the
function $\delta:S_n\to S_0$ which simply deletes all occurrences of variables
is an $S_0$-preserving homomorphism. As another example, assume that $n\geq 2$
and define $\mu:S_n\to S_n$ where $\mu(w)$ is obtained from $w$ by replacing each occurrence of $v_2$ by 
$v_1v_2$. Given $\vec x\in \bea^n$, $h_{\vec x}\circ\mu$ is an $S_0$-preserving 
homomorphism which cannot be obtained by composing those of the kind mentioned previously; in fact $|h_{\vec x}\circ\mu (w)|>|w|$ for each $w\in S_n$.

\begin{definition}\label{defsind} Let $S$, $T$, and $R$ be semigroups
such that $S\cup T$ is a semigroup and $T$ is an ideal of
$S\cup T$. Then
a homomorphism $\tau:T\rightarrow R$ is
 said to be {\it $S$-independent\/}
if, for every $w\in T$ and every $u\in S$, 
$\tau(uw)=\tau(w)=\tau(wu)$.
\end{definition}

In most cases, the above definition will be applied to the case $S=S_0$ and $T=S_n$ for some $n\in \ben$. 
We shall see later in Lemma \ref{Ab} that if $n\in\ben$, $R$ is a cancellative commutative 
semigroup, and $\tau:S_n \to R$ is an $S_0$-independent  homomorphism, then 
$\tau(w)=\tau(w')$ whenever $|w|_{v_i}=|w'|_{v_i}$ for each $i\in\nhat{n}$. 
For reasons which will be made clear in Section \ref{sechomo},
we will primarily be concerned with $S_0$-independent
homomorphisms from $S_n$ to $(\ben,+)$ of the form
$\tau(w)=|w|_{v_i}$ for some $i\in\nhat{n}$.

\begin{lemma}\label{phi} Let $S$ and $T$ be semigroups
such that $S\cup T$ is a semigroup and $T$ is an ideal of
$S\cup T$. Let $\phi: T\to C$ be an $S$-independent homomorphism from $T$ into the 
topological center of a compact right topological semigroup $C$. Then $\phi$ extends to a
continuous homomorphism from $\beta T$ into $C$, which we shall also denote by $\phi$.   For every 
$q\in \beta T$ and every $p\in \beta S$, 
$ \phi(q)= \phi(pq)= \phi(qp)$. \end{lemma}

\begin{proof} The fact that $\phi$ extends to a continuous homomorphism is \cite[Corollary 4.22]{HS}.
Let $p\in\beta S$ and $q\in\beta T$ be given.
 In the following expressions let $s$ and $t$ denote members
of $S$ and $T$ respectively. Since $\phi$ is continuous on $\beta T$ and since
both $pq$ and $qp$ are in $\beta T$ by \cite[Corollary 4.18]{HS}, we 
have that \[\phi(pq)=\displaystyle\phi(\lim_{s\to p}\lim_{t\to q}st)=\lim_{s\to p}\lim_{t\to q}\phi(st)=\lim_{t\to q}\phi(t)= {\phi}(q)\] and similarly
\[\phi(qp)=\displaystyle\phi(\lim_{t\to q}\lim_{s\to p}ts)=\lim_{t\to q}\lim_{s\to p}\phi(ts)=\lim_{t\to q}\phi(t)= {\phi}(q).\] \end{proof}

\begin{theorem}\label{noncom} Let $S$ and $T$ be semigroups
such that $S\cup T$ is a semigroup and $T$ is an ideal of
$S\cup T$.  Let $\phi:T\to C$ be an $S$-independent homomorphism  
from $T$ into a compact right topological 
semigroup $C$ with $\phi[T]$ contained in the topological center of $C$ and denote
also by $\phi$ its continuous extension to $\beta T$.   
Let $F$ be a finite nonempty set of homomorphisms from $S\cup T$ into $S$ which are each equal to the identity on $S$,
and let $D$ be a piecewise syndetic subset of $S$. Let $p$ be an idempotent in 
$ \phi[\beta T]$, and let $U$ be a neighborhood of $p$ in $C$. There exists 
$w\in T$ such that $\phi(w)\in U$ and $\nu(w)\in D$ for every $\nu\in F$. \end{theorem}

\begin{proof}  Since $D$ is piecewise syndetic in $S$, pick by 
\cite[Theorem 4.43]{HS} some $s\in S$ such that $s^{-1}D$ is central in $S$ and pick
a minimal idempotent $r\in\beta S$ such that $s^{-1}D\in r$.

Let $V=\phi^{-1}[\{p\}]$. Since $\phi$ is a continuous homomorphism from $\beta T$ to $C$,
$V$ is a compact subsemigroup of $\beta T$. By Lemma \ref{phi}, $Vr$ is a left ideal
of $V$ and $rV$ is a right ideal of $V$. Pick an idempotent $q\in Vr\cap rV$ and note
that $q\leq r$ in $\beta T$. By Theorem \ref{blass}(1), $\nu(q)=r$ for every $\nu\in F$.

Since $s^{-1}D\in r$ we have that for each $\nu\in F$, $\nu^{-1}[s^{-1}D]\in q$. 
Since $U$ is a neighborhood of $p$, pick
$R\in q$ such that $\phi[\,\overline {R}\,]\subseteq U$.
Pick $w\in R\cap\bigcap_{\nu\in F}\nu^{-1}[s^{-1}D]$.
Then $\phi(sw)=\phi(w)\in U$ and for $\nu \in F$, 
$\nu(w)\in s^{-1}D$ so $\nu(sw)=s\nu(w)\in D$.\end{proof}

We obtain the first result that was stated in the abstract as a corollary to 
Theorem \ref{noncom}.

\begin{corollary}\label{corabst} Define $\tau:S_1\to\ben$ by
$\tau(w)=|w|_{v_1}$, let $S_0$ be finitely colored, and let
$\langle x_n\rangle_{n=1}^\infty$ be a sequence in $\ben$.
There exists $w\in S_1$ such that $\{w(a):a\in \A\}$ is monochromatic
and $\tau(w)\in FS(\langle x_n\rangle_{n=1}^\infty)$.\end{corollary}

\begin{proof} Let $S=S_0$, let $T=S_1$, and let $C=\beta\ben$. Then
$\tau[S_1]$ is contained in the topological center of $C$. Denote also
by $\tau$ the continuous extension taking $\beta S_1$ to $\beta\ben$.
Given $a\in \A$, define $f_a:S_0\cup S_1\to S_0$ by
$$f_a(w)=\left\{\begin{array}{cl} w(a)&\hbox{if }w\in S_1\\
w&\hbox{if }s\in S_0\,,\end{array}\right.$$
and let $F=\{f_a:a\in \A\}$. Then $F$ is 
a finite nonempty set of homomorphisms from $S_0\cup S_1$ into $S_0$ 
which are each equal to the identity on $S_0$. Pick by \cite[Lemma 5.11]{HS}
an idempotent $p\in\beta\ben$ such that $FS(\langle x_n\rangle_{n=1}^\infty)\in p$.
Pick any $q\in K(\beta S_0)$ and pick $D\in q$ which is monochromatic. Note
that $\tau[S_1]=\ben$ so by \cite[Exercise 3.4.1]{HS}, $\tau[\beta S_1]=\beta\ben$.
Therefore, $p\in \tau[\beta S_1]$. Consequently, Theorem \ref{noncom}
applies with $U=\overline{FS(\langle x_n\rangle_{n=1}^\infty)}$.
\end{proof}

\begin{corollary}\label{cornoncom} Let $n\in\ben$. Let 
$\phi:S_n\to C$ be an $S_0$-independent homomorphism from $S_n$ into a compact right topological 
semigroup $C$ with $\phi[S_n]$ contained in the topological center of $C$ and denote
also by $\phi$ the continuous extension to $\beta S_n$.  
Let $F$ be a finite nonempty set of $S_0$-preserving homomorphisms from $S_n$ into $S_0$,
let $D$ be a piecewise syndetic subset of $S_0$, let $p$ be an idempotent in 
$ \phi[\beta S_n]$, and let $U$ be a neighborhood of $p$ in $C$. There exists 
$w\in S_n$ such that $\phi(w)\in U$ and $\nu(w)\in D$ for every $\nu\in F$. \end{corollary}

\begin{proof} Let $S=S_0$, let $T=S_n$, and for $\nu\in F$, extend 
$\nu$ to $S_0\cup S_n$ by defining $\nu$ to be the identity on $S_0$.
Then Theorem \ref{noncom} applies.\end{proof}

\begin{lemma}\label{p in image} Let $(T,\cdot)$ be a discrete
semigroup and let $m,n\in\ben$. Let $\phi: S_n\to \bigtimes_{i=1}^m T$ be an $S_0$-independent homomorphism. 
Then $\phi$ extends to a continuous $S_0$-independent homomorphism $\phi: \beta S_n\to \bigtimes_{i=1}^m\beta T$. 
Moreover if  $\vec p=(p_1,p_2,\ldots, p_m)$ is an idempotent in $\bigtimes_{i=1}^m\beta T$ with the property that
whenever $B_i\in p_i$ for each $i\in\nhat{m}$, there exists $w\in S_n$  such that $\phi(w)\in \bigtimes_{i=1}^m B_i$, then
$\vec p \in \phi[\beta S_n]$.
\end{lemma}

\begin{proof}Let $C=\bigtimes_{i=1}^m\beta T$.  Regarding $\phi$ as 
an $S_0$-independent homomorphism from $S_n$ into the right topological 
semigroup $C$, we see that $\phi[S_n]$ is contained in  $ \bigtimes_{i=1}^m T$ which 
in turn is contained in the topological center of $C$ by \cite[Theorem 2.22]{HS}.  
Hence by \cite[Corollary 4.22]{HS}, $\phi$ extends to a continuous  
homomorphism from $\beta S_n$ into $C$. To see that the extension is
$S_0$-independent, let $u\in S_0$ and let $p\in \beta S_n$. Then, letting
$s$ denote a member of $S_n$, we have
$$\phi(up)=\phi(\lim_{s\to p}us)=\lim_{s\to p}\phi(us)=\lim_{s\to p}\phi(s)=\phi(\lim_{s\to p}s)=\phi(p)$$
and similarly, $\phi(pu)=\phi(p)$.

Now assume that $\vec p=(p_1,p_2,\ldots, p_m)$ is an idempotent in $\bigtimes_{i=1}^m\beta T$ and
whenever $B_i\in p_i$ for each $i\in\nhat{m}$, there exists $w\in S_n$  such that $\phi(w)\in \bigtimes_{i=1}^m B_i$,
To see that $\vec p\in \phi[\beta S_n]$ let $(B_1,B_2,\ldots,B_m)\in\bigtimes_{i=1}^m p_i$, and let 
$$G_{(B_1,\ldots,B_m)}=\{w\in S_n:\phi(w)\in\bigtimes_{i=1}^m B_i\}\,.$$
Then by assumption, ${\mathcal G}=\{G_{(B_1,\ldots,B_m)}:(B_1,B_2,\ldots,B_m)\in\bigtimes_{i=1}^m p_i\}$
has the finite intersection property so one may pick $q\in\beta S_n$ such that
${\mathcal G}\subseteq q$. Then $\vec p=\phi(q)\in \phi[\beta S_n]$. 
\end{proof}

\begin{theorem}\label{thnoncom} Let $(T,\cdot)$ be a 
discrete semigroup and let $m,n\in\ben$.
Let $\vec p=(p_1,p_2,\ldots,
p_m)$ be an idempotent
in $\bigtimes_{i=1}^m\beta T$. For $i\in\nhat{m}$
let $\tau_i$ be an $S_0$-independent homomorphism from $S_n$ to $T$.  
Assume that whenever $B_i\in p_i$ for each
$i\in\nhat{m}$, there exists $w\in S_n$ such that
$\big(\tau_1(w),\tau_2(w),\ldots,\tau_m(w)\big)\in\bigtimes_{i=1}^m B_i$.
Let $D$ be a piecewise syndetic subset of $S_0$ and let
$F$ be a finite nonempty set of $S_0$-preserving homomorphisms
from $S_n$ to $S_0$. Then whenever
$B_i\in p_i$ for each $i\in\nhat{m}$, there exists
$w\in S_n$ such that $\nu(w)\in D$ for each $\nu\in F$
and for each $i\in\nhat{m}$, $\tau_i(w)\in B_i$.
\end{theorem}

\begin{proof} Define $\phi:S_n\to \bigtimes_{i=1}^m T$ by $$\phi(w)=\big(\tau_1(w),\tau_2(w),\ldots,\tau_m(w)\big)\,.$$
Then $\phi$ is an $S_0$-independent homomorphism and hence 
by Lemma~\ref{p in image}, $\phi$ extends to a continuous $S_0$-independent homomorphism $\phi: \beta S_n\to \bigtimes_{i=1}^m\beta T$ and 
$\vec p\in\phi[\beta S_n]$. The result now follows from Corollary \ref{cornoncom}.
\end{proof}

\begin{corollary}\label{cornoncompl}
 Let $k,n\in\ben$ with $k<n$
and let $T$ be the set of words over $\{v_1,v_2,\ldots,v_k\}$
in which $v_i$ occurs for each $i\in\nhat{k}$. Given
$w\in S_n$ let $\tau(w)$ be obtained from $w$ by deleting all occurrences of
elements of $\A$ as well as all occurrences of $v_i$
for $k<i\leq n$ deleted. Let $\langle y_t\rangle_{t=1}^\infty$
be a sequence in $T$, let $F$ be a finite nonempty set of $S_0$-preserving homomorphisms
from $S_n$ to $S_0$, and let $D$ be a piecewise syndetic subset
of $S_0$. There exists $w\in S_n$ such that
$\nu(w)\in D$ for all $\nu\in F$ and
$\tau(w)\in FP(\langle y_t\rangle_{t=1}^\infty)$.
\end{corollary}

\begin{proof} As noted in the introduction, we can 
pick an idempotent $p\in\beta T$ such that
$FP(\langle y_t\rangle_{t=1}^\infty)\in p$. Since
$\tau$ is an $S_0$-independent homomorphism from $S_n$ onto $T$,
Theorem \ref{thnoncom} applies with $m=1$.\end{proof}

Theorem \ref{matrix} involves a matrix with entries from $\beq$ or $\bez$.
In order to ensure that matrix multiplication is distributive, 
we assume that $T$ is commutative and write
the operation as $+$.

\begin{theorem}\label{matrix} 
 Let $(T,+)$ be a commutative semigroup, let $k,m,n\in\ben$,
and let $M$ be a $k\times m$ matrix. If $T$ is not cancellative
assume that the entries of $M$ come from $\omega$. If $T$ is 
isomorphic to a subsemigroup of a direct sum of copies of 
$(\beq,+)$ (so that multiplication by members of $\beq$ makes sense),
assume that the entries of $M$ come from $\beq$. Otherwise
assume that the entries of $M$ come from $\bez$.
For $i\in\nhat{m}$
let $\tau_i$ be an $S_0$-independent homomorphism from $S_n$ to $T$.  
Define a function $\psi$ on $S_n$ by
$\psi(w)=\left(\begin{array}{c}\tau_1(w)\\
\tau_2(w)\\
\vdots\\
\tau_m(w)\end{array}\right)$. Let $\vec p=(p_1,p_2,\ldots,p_k)$ be an idempotent
in $\bigtimes_{i=1}^k\beta T$ with the property that
whenever $B_i\in p_i$ for each $i\in \nhat{k}$, 
there exists $\vec z\in \psi[S_n]$ such that
$M\vec z\in\bigtimes_{i=1}^k B_i$.
Let $F$ be a finite nonempty set of $S_0$-preserving homomorphisms from $S_n$ to $S_0$
and let $D$ be a piecewise syndetic subset of $S_0$. Then whenever
$B_i\in p_i$ for each $i\in\nhat{k}$, there exists
$w\in S_n$ such that $\nu(w)\in D$ for every $\nu\in F$
and $M\psi(w)\in\bigtimes_{i=1}^kB_i$.
\end{theorem}

\begin{proof} If $T$ is not cancellative, let $G=T$. If $T$ is isomorphic
to a subsemigroup of $\bigoplus_{i\in I}\beq$ for some set $I$, assume that 
$T\subseteq\bigoplus_{i\in I}\beq$ and let $G=\bigoplus_{i\in I}\beq$. Otherwise
let $G$ be the group of differences of $T$. In each case we
define an $S_0$-independent  homomorphism $\phi:S_n\to \bigtimes_{j=1}^k G$. by $\phi(w)=M\psi(w)$.
Let $C=\bigtimes_{j=1}^k\beta G$. Then by Lemma~\ref{p in image}, $\phi$ extends to an $S_0$-independent homomorphism 
$\phi: \beta S_n\to C$ and  $\vec p\in\phi[\beta S_n]$. The rest now follows from  Corollary \ref{cornoncom}.\end{proof}

\begin{definition}\label{defmu} Let $n\in\ben$ and let $j\in\nhat{n}$.
Define $\mu_j:S_n\to\ben$ by $\mu_j(w)=|w|_{v_j}$.\end{definition}

The following corollary provides sufficient conditions for
applying Theorem \ref{matrix}.

\begin{corollary}\label{lowertriag} 
Let $m,n\in\ben$ with $m\leq n$. Let
$M$ be an $m\times m$ lower triangular matrix with rational
entries. Assume that the entries on the diagonal are positive and the entries
below the diagonal are negative or zero. Let $\vec p=(p_1,p_2,\ldots,p_m)$
be an idempotent in $\bigtimes_{i=1}^m\beta\ben$.
For $i\in\nhat{m}$ let
$\tau_i=\sum_{j=1}^n\alpha_{i,j}\mu_j$ where each $\alpha_{i,j}\in\beq$.
Assume that for each $i\in\nhat{m}$ we can choose $t(i)\in\nhat{n}$
such that $\alpha_{i,t(i)}>0$, if $l\in\nhat{m}$ and 
$l>i$, then $\alpha_{i,t(l)}=0$, and if $l\in\nhat{m}$
and $l<i$, then $\alpha_{i,t(l)}\leq 0$. Then each $\tau_i$ is an $S_0$-independent
homomorphism from $S_n$ to $\beq$. Let $F$ be a nonempty finite set of
$S_0$-preserving homomorphisms from $S_n$ to $S_0$ and let 
$D$ be a piecewise syndetic subset of $S_0$. Whenever
$B_i\in p_i$ for each $i\in\nhat{m}$, there exists
$w\in S_n$ such that $\nu(w)\in D$ for each $\nu\in F$
and $$M\left(\begin{array}{c}\tau_1(w)\\ \tau_2(w)\\
\vdots\\ \tau_m(w)\end{array}\right)\in\bigtimes_{i=1}^mB_i\,.$$
\end{corollary}

\begin{proof} Define $\psi:S_n\to \beq^m$ by $\psi(w)=\left(\begin{array}{c}\tau_1(w)\\
\vdots\\ \tau_m(w)\end{array}\right)$. We wish to apply Theorem \ref{matrix}
with $T=\beq$. For this we need to show that whenever $B_i\in p_i$ for
$i\in\nhat{m}$, there exists $\vec z\in\psi[S_n]$ such that $M\vec z\in\bigtimes_{i=1}^m B_i$.
So let $B_i\in p_i$ for $i\in\nhat{m}$.

We show first that for each $r\in\ben$, there exists $\vec z\in(r\ben)^m$ such that
$M\vec z\in\bigtimes_{i=1}^m B_i$, so let $r\in\ben$ be given.
Note that $M^{-1}$ is lower triangular with positive
diagonal entries and nonnegative entries below the
diagonal. (Probably the easiest way to see this
is to solve the system of equations $M\vec z=\vec x$
by back substitution. Alternatively we may write $M=D(I+N)$ where $D$ is diagonal with positive 
entries and $N$ is a strictly lower triangular matrix (all of whose non-zero entries are negative) 
verifying $N^m=O$. Setting $x=-N$ in $1-x^m=(1-x)(1+x+x^2+\cdots x^{m-1})$ gives 
$(I+N)^{-1}=I+\sum_{j=1}^{m-1}(-1)^jN^j$. Hence $(I+N)^{-1}$ is lower triangular with 
$1$s along the diagonal and nonnegative entries below the diagonal. 
Multiplying $(I+N)^{-1}$ by $D^{-1}$ on the right gives the desired result.)
Let $c\in\ben$ be such that
all entries of $cM^{-1}$ are nonnegative integers.
By \cite[Lemma 6.6]{HS} $rc\ben\in p_i$ for each $i\in\nhat{m}$
so pick $x_i\in B_i\cap rc\ben$. Letting
$\vec z=M^{-1}\vec x$ one has that $\vec z\in(r\ben)^m$
and $M\vec z\in\bigtimes_{i=1}^mB_i$.

Now assume we have chosen $t(i)$ for $i\in\nhat{m}$ as in the 
statement of the corollary. Pick $d\in\ben$ such that
$d\alpha_{i,j}\in\bez$ for each $i\in\nhat{m}$ and each $j\in\nhat{n}$
and let $\delta_{i,j}=d\alpha_{i,j}$.
 Let $$J=\nhat{n}\setminus\{t(1),t(2),\ldots,t(m)\}\,.$$ Let
$s=\prod_{i=1}^m\delta_{i,t(i)}$ and pick
$r\in\ben$ such that $s$ divides $r$ and 
$$\textstyle r>\max\big\{s\sum_{j\in J}|\delta_{i,j}|:i\in\nhat{m}\big\}\,.$$
Pick $\vec z\in(r\ben)^m$ such that
$M\vec z\in\bigtimes_{i=1}^m B_i$. We shall produce $w\in S_n$ such that
$\psi(w)=\vec z$ by determining $\mu_{j}(w)$ for each $j\in\nhat{n}$.
(To be definite, we then let $w=\prod_{j=1}^nv_j^{\mu_j(w)}$.)

For $j\in J$, let $\mu_j(w)=s$. Let 
$$\mu_{t(1)}(w)=d\frac{z_1}{\delta_{1,t(1)}}-{\textstyle \sum_{j\in J}}
\delta_{1,j}\frac{s}{\delta_{1,t(1)}}$$ and note that $\prod_{l=2}^m \delta_{l,t(l)}$ divides
$\mu_{t(1)}(w)$ and by the choice of $r$, $\mu_{t(1)}(w)>0$, as is, of course, required.
Now let $k\in\{2,3,\ldots,m\}$ and assume that for each
$i\in\nhat{k-1}$, we have chosen $\mu_{t(i)}(w)\in\ben$ such that
$\sum_{l=i+1}^m\delta_{l,t(l)}$ divides $\mu_{t(i)}(w)$. Then let
$$ \mu_{t(k)}(w)=d\frac{z_k}{\delta_{k,t(k)}}-
{\textstyle\sum_{i=1}^{k-1}}\frac{\delta_{k,t(i)}}{\delta_{k,t(k)}}\mu_{t(i)}(w)
-{\textstyle\sum_{j\in J}}s\frac{\delta_{k,j}}{\delta_{k,t(k)}}\,.$$
Then $\mu_{t(k)}(w)\geq \frac{1}{\delta_{k,t(k)}}(dz_k-
\sum_{j\in J}s\delta_{k,j})>0$ and, if $k<m$, then
$\sum_{l=k+1}^m\delta_{l,t(l)}$ divides $\mu_{t(k)}(w)$.

It is now a routine matter to verify that for
$k\in\nhat{m}$, $\tau_k(w)=\sum_{i=1}^k\alpha_{k,t(i)}\mu_{t(i)}(w)+\sum_{j\in J}
\alpha_{k,j}\mu_j(w)=z_k$.
\end{proof}

The sufficient conditions in Corollary \ref{lowertriag} on
the coefficients $\alpha_{i,j}$ of the homomorphisms $\tau_i$ apply to
all lower triangular matrices with positive
diagonal entries and entries below the diagonal
less than or equal to zero.  A complete solution
to the problem of which matrices and which $S_0$-independent
homomorphisms satisfy the hypotheses of Theorem \ref{matrix}
seems quite difficult. The following simple example illustrates
that one cannot get necessary and sufficient conditions on the 
coefficients of the homomorphisms $\tau_i$
valid for all lower triangular matrices with positive
diagonal entries and entries below the diagonal
less than or equal to zero. 

\begin{theorem}\label{example} Let $M=\left(\begin{array}{cc}1&0\\
-1&1\end{array}\right)$, let $N=\left(\begin{array}{cc}1&0\\
0&1\end{array}\right)$, let $\tau_1=2\mu_1+\mu_2$ and let
$\tau_2=\mu_1+2\mu_2$.
\begin{itemize}
\item[(1)] If $p_1$ and $p_2$ are any idempotents in $\beta\ben$,
$B_1\in p_1$, and $B_2\in p_2$, $F$ is a finite set of 
$S_0$-preserving homomorphisms from $S_2$ to $S_0$, and 
$D$ is a piecewise syndetic subset of $S_0$, then there exists $w\in S_2$
such that $M\left(\begin{array}{c}\tau_1(w)\\
\tau_2(w)\end{array}\right)\in\bigtimes_{i=1}^2B_i$ and 
$\nu(w)\in D$ for each $\nu\in F$.
\item[(2)] There exist idempotents $p_1$ and $p_2$ in $\beta\ben$ and
sets $B_1\in p_1$ and $B_2\in p_2$ for which there does not exist
$w\in S_2$
such that $N\left(\begin{array}{c}\tau_1(w)\\
\tau_2(w)\end{array}\right)\in\bigtimes_{i=1}^2B_i$.
\end{itemize}\end{theorem}

\begin{proof} (1) Let $p_1$ and $p_2$ be idempotents in $\beta\ben$,
and let $B_1\in p_1$ and $B_2\in p_2$ be given. By Theorem 
\ref{thnoncom}, it suffices to show that there exists
$w\in S_2$
such that $M\left(\begin{array}{c}\tau_1(w)\\
\tau_2(w)\end{array}\right)\in\bigtimes_{i=1}^2B_i$
By \cite[Lemma 6.6]{HS},
$3\ben\in p_1$ and $3\ben\in p_2$. Pick $z_2\in B_2\cap 3\ben$ and pick
$z_1>z_2$ in $B_1\cap 3\ben$. Let $k_1= \frac{1}{3}z_1-\frac{1}{3}z_2$ and let
$k_2= \frac{1}{3}z_1-\frac{1}{3}z_2$. Let $w=v_1^{k_1}v_2^{k_2}$ so that
$\mu_1(w)=\frac{1}{3}z_1-\frac{1}{3}z_2$
and $\mu_2(w)=\frac{1}{3}z_1+\frac{2}{3}z_2$. Then 
$\tau_1(w)=z_1$, $\tau_2(w)=z_1+z_2$, and
$M\left(\begin{array}{c}\tau_1(w)\\
\tau_2(w)\end{array}\right)=\left(\begin{array}{c}z_1\\ z_2\end{array}\right)$.

(2) Let $B_1=FS(\langle 2^{4n}\rangle_{n=1}^\infty)$ and let
$B_2=FS(\langle 2^{4n+2}\rangle_{n=1}^\infty)$. By \cite[Lemma 5.11]{HS} pick
idempotents $p_1$ and $p_2$ in $\beta\ben$ such that $B_1\in p_1$ and $B_2\in
p_2$. Suppose we have some $w\in S_2$ and elements $z_1\in B_1$ and 
$z_2\in B_2$ such that $N\left(\begin{array}{c}\tau_1(w)\\
\tau_2(w)\end{array}\right)=\left(\begin{array}{c}z_1\\ z_2\end{array}\right)$.
Then $2z_1-z_2=3\mu_1(w)>0$ and $2z_2-z_1=3\mu_2(w)>0$ so
$z_2<2z_1$ and $z_1<2z_2$. Pick $F,G\in\pf(\ben)$ such that
$z_1=\sum_{t\in F}2^{4t}$ and $z_2=\sum_{t\in G}2^{4t+2}$.
Let $m=\max F$ and let $k=\max G$. Then
$2^{4m}\leq z_1<2^{4m+1}$ and $2^{4k+2}\leq z_2<2^{4k+3}$.
Then $2^{4m+2}>2z_1>z_2\geq 2^{4k+2}$ so $m\geq k+1$.
Also $2^{4k+4}>2z_2>z_1\geq 2^{4m}\geq 2^{4k+4}$, a contradiction.
\end{proof}

A $k\times m$ matrix $M$ is {\it image partition regular over $\ben$\/}
if and only if, whenever $\ben$ is finitely colored, there is some 
$\vec z\in\ben^m$ such that the entries of $M\vec z$ are
monochromatic.  This class includes all triangular (upper or lower)
matrices with rational entries and positive diagonal entries.
See \cite[Theorem 15.24]{HS} for half an alphabet of
characterizations of matrices that are image partition regular over $\ben$.

\begin{corollary}\label{IPR} Let $k, m,n\in\ben$ with $m\leq n$. Let
$M$ be a $k\times m$ matrix with rational entries which is image
partition regular over $\ben$. 
Let $p$ be a minimal idempotent in $\beta\ben$ and
let $\widehat{p}=(p,p,\ldots,p)\in \bigtimes_{i=1}^k \beta\ben$.
Let $\sigma$ be an injection
from $\nhat{m}$ to $\nhat{n}$.
For $i\in\nhat{m}$ let $\tau_i=\mu_{\sigma(i)}$.
Let $F$ be a nonempty finite set of $S_0$-preserving homomorphisms from $S_n$ to $S_0$ and
let $D$ be a piecewise syndetic subset of $S_0$. Then whenever
$B\in p$ there exists
$w\in S_n$ such that $\nu(w)\in D$ for each $\nu\in F$ and 
$$M\left(\begin{array}{c}\tau_1(w)\\ \tau_2(w)\\
\vdots\\ \tau_m(w)\end{array}\right)\in\bigtimes_{i=1}^kB\,.$$
\end{corollary}

\begin{proof} We note
that the mapping $$w\mapsto\big(\tau_1(w),\tau_2(w),\ldots,\tau_m(w)\big)$$
defines an $S_0$-independent homomorphism from $S_n$ onto $\ben^m$. So in order to apply Theorem \ref{matrix},
we must verify that whenever 
$B_i\in p$ for each $i\in\nhat{k}$, there exists
$\vec z\in \ben^m$ such that $M\vec z\in\bigtimes_{i=1}^k B_i$. We then
pick $w\in S_n$ such that $\tau_i(w)=z_i$ for each $i\in\nhat{m}$, which
one may do because $\sigma$ is injective.

Now $\bigcap_{i=1}^k B_i\in p$ so $B=\bigcap_{i=1}^k B_i$ is central
in $\ben$. By \cite[Theorem 15.24(h)]{HS} there exists $\vec z\in\ben^m$
such that $M\vec z\in B^k$. \end{proof}

Corollary \ref{IPR} applies to a much larger class of matrices
than Corollary \ref{lowertriag}, but is more restrictive
in that the same minimal idempotent must occur in each 
coordinate.  
Suppose we have a $k\times m$ matrix $M$ which is image partition regular over $\ben$. If
we knew that whenever $B_1,B_2,\ldots,B_k$ are central subsets of $\ben$,
there exist $\vec z\in\ben^m$ with $M\vec z\in\bigtimes_{i=1}^k B_i$, then
in Corollary \ref{IPR} we could allow $\vec p=(p_1,p_2,\ldots,p_k)$ to be
an arbitrary minimal idempotent in $\bigtimes_{i=1}^k\beta\ben$. We shall
see now that this fails.

\begin{theorem}\label{centralnogood} Let $M=\left(\begin{array}{cc} 1&1\\ 1&2\end{array}\right)$. 
Then $M$ is image partition regular over $\ben$.  For $x\in\ben$
let $\phi(x)=\max\{t\in\omega:2^t\leq x\}$ and for $i\in\{0,1,2,3\}$
let $B_i=\{x\in\ben:\phi(x)\equiv i\ (\mod\ 4)\}$.
Then $B_0$ and $B_2$ are central and there do not exist
$x$ and $y\in\ben$ such that $M\left(\begin{array}{c}x\\ y\end{array}\right)\in B_0\times B_2$.
\end{theorem}

\begin{proof} By \cite[Theorem 15.5]{HS} $M$ is image partition regular over
$\ben$. 
Since $\ben=\bigcup_{i=0}^3 B_i$ some $B_i$ is central. But then, by
\cite[Lemma 15.23.2]{HS}, each $B_i$ is central. Suppose we have some
$x,y\in\ben$ such that $M\left(\begin{array}{c}x\\ y\end{array}\right)\in B_0\times B_2$.
Let $n=\phi(x+y)$. Then $2^n\leq x+y<2^{n+1}$ so
$y<2^{n+1}-x$ and thus $2y<2^{n+2}-2x$ so
$x+2y<2^{n+2}-x<2^{n+2}$ and thus $\phi(x+2y)\in\{n,n+1\}$.\end{proof}

Note also that Corollary \ref{IPR} is more restrictive than
Corollary \ref{lowertriag} in that the idempotent $p$
is required to be minimal. It is well known and easy to
see that $FS(\langle 2^{2t}\rangle_{t=1}^\infty)$ does not
contain any three term arithmetic progression.  Consequently,
if 
$$M=\left(\begin{array}{cc}1&0\\ 1&1\\ 1&2\end{array}\right)\,,$$
then the assumption in Corollary \ref{IPR} that the idempotent
$p$ is minimal cannot be deleted.

We remarked in the introduction that 
we are not concerned with the instances of Theorem \ref{GRth} with $m>0$ because
the natural versions of our results in this section are not valid. 
The results in this section apply to all piecewise syndetic subsets of $S_0$.
In particular, they apply to central subsets. It was shown in 
\cite[Theorem 3.6]{CHS} that, given $a\in \A$,  there is a central set $D\in S_1$ such that
there is no $w\in S_2$ with $\{w(av_1),w(v_1a)\}\subseteq D$.

\section{Homomorphisms satisfying our hypotheses}\label{sechomo}

In Corollary \ref{lowertriag} we produced $S_0$-independent homomorphisms from $S_n$ to
$\beq$ as linear combinations of the functions $\mu_i$ with
coeffients from $\beq$. We shall see in
Corollary \ref{linear} that if $T$ is commutative and cancellative, then
the only $S_0$-independent homomorphisms $\varphi :S_n\to T$ are of the form
$\varphi(w)=\sum_{i=1}^n\mu_i(w)\cdot a_i$ where each $a_i$ is in the 
group of differences of $T$.

In Corollary \ref{IPR} we used $S_0$-independent homomorphisms
$\tau_i=\mu_{\sigma(i)}$
from $S_n$ to $\ben$ and the surjection $w\mapsto \big(\tau_1(w),\tau_2(w),\ldots,\tau_m(w)\big)$
from $S_n$ onto $\ben^m$.  
We show in Corollary \ref{surj3} that if $T=\ben$, then these are essentially
the only choices for $\tau_i$ satisfying the hypotheses of Theorem \ref{matrix}.

Recall that throughout this section $\A$ is a fixed nonempty finite alphabet.

\begin{definition}\label{defAb} Let $n\in\ben$. 
For $w\in S_n$, let $w'\in\{v_1,v_2,\ldots,v_n\}^+$ be obtained
from $w$ by deleting all occurrences in $w$ of letters belonging to $\A$.
\end{definition}

\begin{lemma}\label{lcanc} Fix $n\in \ben$. Let  $(T,+)$ be a cancellative 
semigroup and let $\varphi: S_n\to T$ be an $S_0$-independent  homomorphism. 
Then $\varphi(w)=\varphi(w')$ for all $w\in S_n$. 
\end{lemma}

\begin{proof}  It suffices to show that if $w_1,w_2\in S_n$ and $u\in S_0$, then 
$\varphi(w_1uw_2)=\varphi(w_1w_2)$.  Let $v=v_1v_2\cdots v_n$. 
On one hand $\varphi(vw_1uw_2v)=\varphi(v) + \varphi(w_1uw_2) + \varphi(v)$, and on the other hand
$\varphi(vw_1uw_2v)=\varphi(vw_1u) +\varphi(w_2v)=\varphi(vw_1) + \varphi(w_2v)=\varphi(vw_1w_2v)=\varphi(v) + \varphi(w_1w_2) + \varphi(v)$.
 The result now follows. \end{proof}

\begin{lemma}\label{Ab} Fix $n\in \ben$. Let  $(T,+)$ be a cancellative and commutative 
semigroup and let $\varphi: S_n\to T$ be an $S_0$-independent  homomorphism. 
For each $w_1,w_2\in S_n$ we have $\varphi(w_1)=\varphi(w_2)$ whenever
$w'_1\sim_{Ab}w'_2$. \end{lemma}

\begin{proof} Assume $w_1,w_2\in S_n$ and $w'_1\sim_{Ab}w'_2$. 
Let $m=|w'_1|=|w'_2|$. We show that $\varphi(w'_1)=\varphi(w'_2)$ which in turn implies that
$\varphi(w_1)=\varphi(w_2)$ by Lemma \ref{lcanc}.  
The result is immediate in case $n=1$ for in this case $w'_1=w'_2=v_1^m$. 
So let us assume that $n\geq 2$ in which case $m\geq 2$. 
Since the symmetric group on $m$-letters is generated by the 
$2$-cycle $(1,2)$ and the $m$-cycle $(1,2,\ldots , m)$ 
it suffices to show
\begin{itemize}
\item[(i)] If $x,w\in  (\A\cup \{v_1,v_2,\ldots ,v_n\})^+$ and $xw\in S_n$,
then $wx \in S_n$ and $\varphi(wx)=\varphi(xw)$.  
\item[(ii)] Let $\varepsilon$ be the empty word. 
If $x,y \in  (\A\cup \{v_1,v_2,\ldots ,v_n\})^+$, $w\in(\A\cup \{v_1,v_2,\ldots ,v_n\})^+\cup\{\varepsilon\}$
and $xyw\in S_n$, then $yxw \in S_n$ and $\varphi(yxw)=\varphi(xyw)$. 
\end{itemize}

Then given $l_1,l_2,\ldots,l_m\in \A\cup \{v_1,v_2,\ldots ,v_n\}$ by (i) we have
$\varphi(l_1l_2\cdots l_m)=\varphi(l_2l_3\cdots l_ml_1)$ and by (ii) we have
$\varphi(l_1l_2l_3\cdots l_m)=\varphi(l_2l_1l_3\cdots l_m)$.

To establish (i), we have $\varphi(xw) + \varphi(xwx)=\varphi(xwxwx)=\varphi(xwx)+ \varphi(wx)$, whence
$\varphi(xw)=\varphi(wx)$. Note that we are using here that $T$ is commutative. 
For (ii), let $v=v_1v_2\cdots v_n$.
Then, using (i) twice,  $\varphi(v)+\varphi(xyw)+\varphi(v)=\varphi(vxywv)=
\varphi(vx)+\varphi(ywv)=\varphi(xv)+\varphi(ywv)=\varphi(xvywv)=
\varphi(xvy)+\varphi(wv)=\varphi(vyx)+\varphi(wv)=\varphi(vyxwv)=
\varphi(v)+\varphi(yxw) +\varphi(v)$. The result now follows. 
\end{proof}

We remark that  Lemma~\ref{Ab} does not hold in general if $T$ is not commutative. 
For example, consider the homomorphism $\varphi :S_3\to S_2$ where $\varphi(w)$ is 
the word in $S_2$ obtained from $w$ by deleting all occurrences of the variable $v_3$ in addition to 
all letters belonging to $\A.$ Then $S_2$ is cancellative and 
$\varphi$ is an $S_0$-independent homomorphism. 
However, $\varphi(v_1v_2v_3)=v_1v_2\neq v_2v_1=\varphi(v_2v_1v_3)$ 
yet $v_1v_2v_3\sim_{Ab}v_2v_1v_3$. 

\begin{theorem}\label{surj1} Fix $n\in \ben$. Let  $(T,+)$ be a cancellative and commutative 
semigroup and let $\varphi: S_n\to T$ be an $S_0$-independent homomorphism. Then 
there exists a homomorphism $f:\ben^n\to T$ such that
$\varphi(w)=f\big(\mu_1(w),\mu_2(w),\ldots,\mu_n(w)\big)$ for all $w\in S_n$.\end{theorem}

\begin{proof} Define $f:\ben^n\to T$ by $f(x_1,x_2,\ldots,x_n)=
\varphi(v_1^{x_1}v_2^{x_2}\cdots v_n^{x_n})$. By Lemma \ref{Ab}, $f$ is as
required.\end{proof}

\begin{corollary}\label{linear} Let $n\in\ben$, let $(T,+)$ be a commutative
and cancellative semigroup, let $G$ be the group of differences of $T$,
and let $\varphi$ be an $S_0$-independent homomorphism from $S_n$ to $T$. There
exist $a_1,a_2,\ldots,a_n$ in $G$ such that for each
$w\in S_n$, $\varphi(w)=\sum_{i=1}^n\mu_i(w)\cdot a_i$.
\end{corollary} 

\begin{proof} Pick a homomorphism $f:\ben^n\to T$ as guaranteed by Theorem \ref{surj1}.
For $j\in\nhat{n}$, define $\vec z^{\,[j]}\in\ben^n$ by, for $i\in\nhat{n}$,
\begin{equation}\label{eq1}z^{\,[j]}_i=\left\{\begin{array}{cl}2&\hbox{if }i=j\\
1&\hbox{if }i\neq j\,.\end{array}\right.\end{equation} Let $\widehat 1=(1,1,\ldots,1)\in\ben^n$.
Let $c=f(\,\widehat 1\,)$ and for $j\in\nhat{n}$, let $a_j=f(\vec z^{\,[j]})-c$.
Then \[(n+1)\cdot c=f(n+1,n+1,\ldots,n+1)=\sum_{j=1}^nf(\vec z^{\,[j]})=(\sum_{j=1}^n a_j)+n\cdot c\]
so $c=\sum_{j=1}^n a_j$. 

We claim that $f(x_1,x_2,\ldots,x_n)=\sum_{j=1}^n x_j\cdot a_j$ for all $(x_1,x_2,\ldots,x_n)\in \ben^n.$ 
To see this we proceed by induction on $\sum_{j=1}^n x_j.$ If $\sum_{j=1}^n x_j=n$ then $(x_1,x_2,\ldots,x_n)=\widehat 1$
whence \[\textstyle f(x_1,x_2,\ldots,x_n)=c=\sum_{j=1}^n a_j=\sum_{j=1}^n 1\cdot a_j\,.\] Next let $N\geq n$ and suppose that
$f(x_1,x_2,\ldots,x_n)=\sum_{j=1}^n x_j\cdot a_j$ for all $(x_1,x_2,\ldots,x_n)\in \ben^n$ with $\sum_{j=1}^nx_j\leq N.$
Let $(x_1,x_2,\ldots,x_n)\in \ben^n$ be such that $\sum_{j=1}^nx_j= N+1.$
Pick $j\in\nhat{n}$ such that $x_j\geq 2$. Then
$$f(x_1,x_2,\ldots,x_n)+f(\,\widehat{1}\,)=
f(x_1,\ldots,x_{j-1},x_j-1,x_{j+1},\ldots,x_n)+f(\vec z^{\,[j]}).$$ 
Since $f(\vec z^{\,[j]})-f(\,\widehat{1}\,)=a_j$ it follows by our induction hypothesis that

\[\textstyle f(x_1,x_2,\ldots,x_n)=\sum_{i=1}^{j-1}x_i\cdot a_i+
(x_j-1)\cdot a_j+ \sum_{i=j+1}^nx_i\cdot a_i +a_j
=\sum_{i=1}^n x_i\cdot a_i\,.
\]
Consequently, for all $w\in S_n$, 
\[\textstyle \varphi(w)=f\big(\mu_1(w),\mu_2(w),\ldots,\mu_n(w)\big)
=\sum_{j=1}^n\mu_j(w)\cdot a_j\,.\]
\end{proof}

In the proof of the next lemma, we shall use the fact that
if $n\in\ben$,\break $f:\ben^n\to\ben$ is a homomorphism, 
$\vec x,\vec y^{\,[1]},\vec y^{\,[2]},\ldots,\vec y^{\,[n]}\in\ben^n$, 
$\alpha_1,\alpha_2,\ldots,\alpha_n\in \bez$, and $\vec x=
\sum_{i=1}^n\alpha_i \vec y ^{\,[i]}$, then
\[\textstyle f(\vec x)=\sum_{i=1}^n\alpha_i f(\vec y ^{\,[i]})\,.\] We note that if $\alpha_i\leq 0$, then
$f$ is not defined at $\alpha_i\vec y^{\,[i]}$.
So to verify the above equality, let $I=\{i\in\nhat{n}:\alpha_i<0\}$ and let
$J=\{i\in\nhat{n}:\alpha_i>0\}$. Then
\[\textstyle \vec x+\sum_{i\in I}(-\alpha_i)\vec y^{\,[i]}=\sum_{i\in J}\alpha_i\vec y^{\,[i]}\] so
\[\textstyle f(\vec x)+\sum_{i\in I}(-\alpha_i)f(\vec y^{\,[i]})=\sum_{i\in J}\alpha_i f(\vec y^{\,[i]})\] so
\[\textstyle f(\vec x)=\sum_{i\in I\cup J}\alpha_i f(\vec y^{\,[i]})=\sum_{i=1}^n\alpha_i f(\vec y^{\,[i]})\,.\]

\begin{lemma}\label{surj0} Let $n\in \ben$ and 
$f: \ben^n \to \ben$ be a surjective homomorphism. Then there 
exists $i\in\nhat{n}$ such that
$f(\vec x)=x_i$ for each $\vec x =(x_1,x_2,\ldots,x_n)\in \ben^n$, i.e., 
$f$ is the projection onto the $i$'th coordinate.  \end{lemma}

\begin{proof} We begin by showing that $f(\,\widehat{1}\,)=1$ where $\widehat{1}=(1,1,\ldots,1)$. 
Since $f$ is surjective, it suffices to show that $f(\vec x)\geq f(\,\widehat{1}\,)$ for each 
$\vec x =(x_1,x_2,\ldots,x_n)\in \ben^n$. For each $r\in \ben$ we have that 
$$r\vec x=(r-1)\widehat{1} +\big(1+ r(x_1-1),1+r(x_2-1),\ldots,1+r(x_n-1)\big)\,.$$
It follows that $rf(\vec x)=f(r\vec x)> f\big((r-1)\widehat{1}\,\big)=(r-1)f(\,\widehat{1}\,)$ or 
equivalently that $r\big(f(\vec x)-f(\widehat {1})\big)>-f(\,\widehat{1}\,)$.
As $r$ is arbitrary we deduce that $f(\vec x)- f(\,\widehat{1}\,)\geq 0$ as claimed.

For each $j\in\nhat{n}$ let $\vec z^{\,[j]}=(z_1^{[j]},z_2^{[j]},\ldots ,z_n^{[j]})\in \ben^n$ be as in (\ref{eq1}). 
As $\sum_{j=1}^n \vec z^{\,[j]}=(n+1)\widehat{1}$ we have
$\sum_{j=1}^nf(\vec z^{\,[j]})=f\big((n+1)\widehat{1})=n+1$.
It follows that there exists a unique $k\in\nhat{n}$ such that $f(\vec z^{\,[k]})=2$ and $f(\vec z^{\,[j]})=1$ for all $j\neq k$.
Without loss of generality, we may assume that $f(\vec z^{\,[1]})=2$ and $f(\vec z^{\,[j]})=1$ for all $j\in\{2,3,\ldots n\}$.

Let $\vec x=(x_1,x_2,\ldots,x_n)\in \ben^n$. We will show that $f(\vec x)=x_1.$
We first note that \[\textstyle (n+1)\vec x=\sum_{i=1}^n(nx_i-\sum_{\stackrel{j=1}{j\neq i}}^n x_j)\vec z^{\,[i]}.\] 
Therefore $$\textstyle(n+1)f(\vec x)=(nx_1-\sum_{j=2}^nx_j)\cdot 2+\sum_{i=2}^n(nx_i-\sum_{\stackrel{j=1}{j\neq i}}^nx_j)\cdot 1$$
and thus $(n+1)f(\vec x)=(n+1)x_1$.	
\end{proof}

\begin{corollary}\label{surj2} Let $n,m\in \ben$. For $i\in\nhat{n}$ and $j\in\nhat{m}$ let $\pi_i:\ben^n\to\ben$ 
and $\pi_j':\ben^m\to\ben$ denote the projections onto the $i$'th and $j$'th coordinates respectively.
Assume that $f:\ben^n\to\ben^m$ is a surjective homomorphism. 
For $i\in\nhat{m}$, let $f_i=\pi_i'\circ f$. 
Then there is an injection
$\sigma:\nhat{m}\to\nhat{n}$ such that for each $i\in\nhat{m}$, $f_i=\pi_{\sigma(i)}$.
In particular $m\leq n$.\end{corollary}

\begin{proof} By hypothesis each $f_i:  \ben^n \to \ben$ is a 
surjective homomorphism. 
Thus by Lemma~\ref{surj0}, there exists a mapping 
$\sigma :  \{1,2,\ldots ,m\}\to  \{1,2,\ldots ,n\}$  such that $f_i(\vec x)=\pi_{\sigma(i)}(\vec x)=
x_{\sigma(i)}$ for each $\vec x \in \ben^n$. But as $f$ is surjective, it follows that $\sigma$ is injective. \end{proof}

\begin{corollary}\label{surj3} Let $n,m \in \ben$. For each $i\in\nhat{m}$ let $\tau_i: S_n \to \ben$ be an 
$S_0$-independent homomorphism. If the mapping $w\mapsto \big(\tau_1(w),\tau_2(w),\ldots,\tau_m(w)\big)$ takes $S_n$ onto $\ben^m$,
then there exists an injection $\sigma: \{1,2,\ldots ,m\} \to \{1,2,\ldots ,n\}$ such that 
$\tau_i=\mu_{\sigma(i)}$ for each $i\in \{1,2,\ldots ,m\}$. In particular we must have $m\leq n$.
\end{corollary}

\begin{proof} By Theorem \ref{surj1}, for each $i\in\nhat{m}$, pick a homomorphism
$f_i:\ben^n\to\ben$ such that $\tau_i(w)=f_i\big(\mu_1(w),\mu_2(w),\ldots,\mu_n(w)\big)$
for each $w\in S_n$.  Define $f:\ben^n\to\ben^m$ by $f(\vec x)=\big(f_1(\vec x),f_2(\vec x),\ldots,f_m(\vec x)\big)$.
We claim that $f$ is surjective, so let $\vec y\in \ben^m$ be given and pick $w\in S_n$ such that
$\big(\tau_1(w),\tau_2(w),\ldots,\tau_m(w)\big)=\vec y$. For $j\in\nhat{n}$, let
$x_j=|w|_{v_j}$. Then $f(\vec x)=\vec y$.  

By Corollary \ref{surj2}, pick an injection
$\sigma:\nhat{m}\to\nhat{n}$ such that for each $i\in\nhat{m}$, $f_i=\pi_{\sigma(i)}$.
Let $w\in S_n$ be given and let $\vec x=(|w|_{v_1},|w|_{v_2},\ldots,|w|_{v_n})$,
Then for $i\in \nhat{m}$, $\tau_i(w)=f_i(\vec x)=x_{\sigma(i)}=|w|_{v_{\sigma(i)}}$.
\end{proof}

\section{Compact subsemigroups of $(\beta\ben)^k$}\label{seccptsemi}

\begin{theorem}\label{P} Let $n\in \ben$, let $C$ be a compact right topological semigroup,
let $\phi:S_n\to C$ be an $S_0$-independent homomorphism for which $\phi[S_n]$
is contained in the topological center of $C$, denote also by $\phi$ the continuous extension
from $\beta S$ to $C$, and let $F$ be a finite nonempty set of $S_0$-preserving
homomorphisms from $S_n$ into $S_0$. Let 
$$\begin{array}{rl}P=\big\{p\in\phi[\beta S_n]:
&\hbox{for every neighborhood }U\hbox{ of }p\\
&\hbox{and every piecewise syndetic subset }D\hbox{ of }S_0\\
&(\exists w\in S_n)\big(\phi(w)\in U\hbox{ and }(\forall \nu\in F)(\nu(w)\in D)\big)\big\}\,.
\end{array}$$
Then $P$ is a compact subsemigroup of $C$ containing all the idempotents 
of $ \phi[\beta S_n]$. \end{theorem}

\begin{proof} It is clear that $P$ is compact. 
By Corollary \ref{cornoncom}, $P$ contains all the idempotents in $ \phi[\beta S_n]$.  
To see that $P$ is a subsemigroup of $C$,  let $p,q\in P$. Let
$U$ be an open neighborhood of $pq$ and let $D$ be a piecewise syndetic 
subset of $S_0$. By \cite[Theorem 4.43]{HS} pick $s\in S_0$ such that
$s^{-1}D$ is central and pick a minimal idempotent $r$ in $\beta S_0$ such that
$s^{-1}D\in r$. Pick a neighborhood $V$ of $p$ such that $\rho_q[V]\subseteq U$.
Since $(s^{-1}D)^\star\in r$, it is piecewise syndetic so pick
$w\in S_n$ such that $\phi(w)\in V$ and $\nu(w)\in(s^{-1}D)^\star$ for
each $\nu\in F$.  

Then $\phi(w)q\in U$ and $\phi(w)$ is in the topological center of $C$
so pick a neighborhood $Q$ of $q$ such that $\lambda_{\phi(w)}[Q]\subseteq U$.
For each $\nu\in F$, $\nu(w)^{-1}(s^{-1}D)^\star\in r$.
Let $E=\bigcap_{\nu\in F}\nu(w)^{-1}(s^{-1}D)^\star$. Then
$E\in r$ so $E$ is piecewise syndetic in $S_0$. Pick
$u\in S_n$ such that $\phi(u)\in Q$ and $\nu(u)\in E$ for each
$\nu\in F$. Then $\phi(swu)=\phi(w)\phi(u)\in U$ and for
each $\nu\in F$, $\nu(swu)=s\nu(w)\nu(u)\in D$.\end{proof}

In the next results we focus on $(\beta\ben)^k$ and $S_0$-independent homomorphisms
from $S_n$ onto $\ben^m$, so by Corollary \ref{surj3} we may assume that we have $m\leq n$
and are dealing with $S_0$-independent homomorphisms $\tau_i$  from $S_n$ to $\ben$ defined
by $\tau_i(w)=|w|_{v_{\sigma(i)}}$ for some injection
$\sigma:\nhat{m}\to\nhat{n}$.  

\begin{definition}\label{defPM} Let $k,m,n\in\ben$ with $m\leq n$, let
$M$ be a $k\times m$ matrix with entries from $\beq$, let
$F$ be a finite nonempty set of $S_0$-preserving homomorphisms from
$S_n$ to $S_0$, and let
$\sigma$ be an injection from $\nhat{m}$ to $\nhat{n}$.
$$\begin{array}{rl} P_{M\!,F\!,\,\sigma}=\{\vec p\in\bigtimes_{i=1}^k\beta\ben:
&\hbox{whenever }D\hbox{ is a piecewise syndetic subset of } S_0\\
&\hbox{and for all }i\in\nhat{k}\,,\,B_i\in p_i\hbox{, there exists}\\
&w\in S_n \hbox{ such that }(\forall\nu\in F)(\nu(w)\in D)\hbox{ and}\\
&M\left(\begin{array}{c}\mu_{\sigma(1)}(w)\\ \vdots\\
\mu_{\sigma(m)}(w)\end{array}\right)\in\bigtimes_{i=1}^k B_i\}\end{array}$$
\end{definition}

Recall that for $\vec x\in\bea^n$ we have defined the $S_0$-preserving homomorphism
$h_{\vec x}:S_n\to S_0$ by $h_{\vec x}(w)=w(\vec x)$.  We are  particularly
interested in the set $\{h_{\vec x}:\vec x\in\bea^n\}$ because of
the relationship with the Hales-Jewett Theorem. We see now that
if $F=\{h_{\vec x}:\vec x\in\A^n\}$, then $P_{M\!,F\!,\,\sigma}$
does not depend on $\sigma$. We keep $\sigma$ in the notation
because there are $S_0$-preserving homomorphisms which are not
of the form $h_{\vec x}$.

\begin{theorem}\label{MFsame} Let $k,m,n\in\ben$ with $m\leq n$, let
$M$ be a $k\times m$ matrix with entries from $\beq$, let
$F=\{h_{\vec x}:\vec x\in \A^n\}$, and let
$\sigma$ and $\eta$ be injections from $\nhat{m}$ to $\nhat{n}$.
Then $P_{M\!,F\!,\,\sigma}=P_{M\!,F\!,\,\eta}$. \end{theorem}

\begin{proof} It suffices to show that $P_{M\!,F\!,\,\sigma}\subseteq P_{M\!,F\!,\,\eta}$,
so let $\vec p\in P_{M\!,F\!,\,\sigma}$. To see that
$\vec p\in P_{M\!,F\!,\,\eta}$, let $D$ be a piecewise syndetic subset of $S_0$ and for
$i\in\nhat{k}$, let $B_i\in p_i$. Pick $w\in S_n$ such that for all $\vec x\in \A^n$, 
$h_{\vec x}(w)\in D$ and $M\left(\begin{array}{c}\mu_{\sigma(1)}(w)\\ \vdots\\
\mu_{\sigma(m)}(w)\end{array}\right)\in\bigtimes_{i=1}^k B_i$.

Define $\delta:\{\sigma(1),\sigma(2),\ldots,\sigma(m)\}\to\nhat{n}$ by,
for $i\in\nhat{m}$, $\delta\big(\sigma(i)\big)=\eta(i)$ and extend
$\delta$ to a permutation of $\nhat{n}$.  Define $w'\in S_n$ by
$w'=w(v_{\delta(1)}v_{\delta(2)}\cdots v_{\delta(n)})$.  Then for
$j\in\nhat{n}$, $\mu_j(w)=\mu_{\delta(j)}(w')$ so for
$i\in\nhat{m}$, $\mu_{\sigma(i)}(w)=\mu_{\delta(\sigma(i))}(w')=\mu_{\eta(i)}(w')$
and thus $$M\left(\begin{array}{c}\mu_{\eta(1)}(w')\\ \vdots\\
\mu_{\eta(m)}(w')\end{array}\right)=M\left(\begin{array}{c}\mu_{\sigma(1)}(w)\\ \vdots\\
\mu_{\sigma(m)}(w)\end{array}\right)\in\bigtimes_{i=1}^k B_i\,.$$

Now let $\vec x\in \A^n$ be given and define $\vec z\in \A^n$ by, 
for $i\in\nhat{n}$, $z_i=x_{\delta(i)}$. Then $h_{\vec x}(w')=h_{\vec z}(w)\in D$.
\end{proof}

If one lets $C=(\beta\ben)^k$ and defines $\phi$ on $S_n$ by
$\phi(w)=M\left(\begin{array}{c}\mu_{\sigma(1)}(w)\\ \vdots\\
\mu_{\sigma(m)}(w)\end{array}\right)$, one may not be able to invoke
Theorem \ref{P} to conclude that $P_{M\!,F\!,\,\sigma}$ is a semigroup
because $\phi$ may not take $S_n$ to $C$. Consider, for example,
$M=\left(\begin{array}{c} 1\\ -1\end{array}\right)$.

\begin{lemma}\label{PMsemi} Let $k,m,n\in\ben$ with $m\leq n$, let
$M$ be a $k\times m$ matrix with entries from $\beq$, and let
$F$ be a finite nonempty set of $S_0$-preserving homomorphisms from
$S_n$ to $S_0$.  Let
$\sigma$ be an injection from $\nhat{m}$ to $\nhat{n}$.
If $P_{M\!,F\!,\,\sigma}\neq\emp$, then $P_{M\!,F\!,\,\sigma}$ is a compact subsemigroup of
$(\beta\ben)^k$.\end{lemma}

\begin{proof} Assume that $P_{M\!,F\!,\,\sigma}\neq\emp$. We begin by showing that $P_{M\!,F\!,\,\sigma}$ is compact. 
Let $\vec p=(p_1,p_2,\ldots,p_k)\in(\beta\ben)^k\setminus P_{M\!,F\!,\,\sigma}$ and pick
piecewise syndetic $D\subseteq S_0$ and $B_i\in p_i$ for each $i\in\nhat{k}$ such that
there is no $w\in S_n$ with $\nu(w)\in D$ for all $\nu\in F$ and 
$M\left(\begin{array}{c}\mu_{\sigma(1)}(w)\\ \vdots\\
\mu_{\sigma(m)}(w)\end{array}\right)\in\bigtimes_{i=1}^k B_i$;
then $\bigtimes_{i=1}^k\overline{B_i}$ is a neighborhood of $\vec p$
which misses $P_{M\!,F\!,\,\sigma}$ so $P_{M\!,F\!,\,\sigma}$ is closed and hence compact.

To see that $P_{M\!,F\!,\,\sigma}$ is a semigroup, let $\vec p,\vec q\in P_{M\!,F\!,\,\sigma}$. 
Let $D$ be a piecewise syndetic subset of $S_0$ and for
each  $i\in\nhat{k}$, let $B_i\in p_i+q_i$.
By \cite[Theorem 4.43]{HS}, pick $s\in S_0$ such that
$s^{-1}D$ is central in $S_0$ and pick a minimal idempotent $r\in\beta S_0$ such that
$s^{-1}D\in r$. For each $i\in\nhat{k}$, let
$C_i=\{x\in\ben:-x+B_i\in q_i\}$ and note that $C_i\in p_i$. 
Then as  $(s^{-1}D)^\star\in r$, we deduce that $(s^{-1}D)^\star$ is central and hence in particular piecewise syndetic.
Since $\vec p\in P_{M\!,F\!,\,\sigma}$, pick $w\in S_n$ such that $\nu(w)\in (s^{-1}D)^\star$ for all $\nu\in F$
and $M\left(\begin{array}{c}\mu_{\sigma(1)}(w)\\ \vdots\\
\mu_{\sigma(m)}(w)\end{array}\right)=\vec z\in\bigtimes_{i=1}^kC_i$.
Let $G=\bigcap_{\nu\in F}\nu(w)^{-1}(s^{-1}D)^\star$.
Then $G\in r$ so $G$ is piecewise syndetic in $S_0$. Also $\vec q\in P_{M\!,F\!,\,\sigma}$
and for each $i\in\nhat{k}$, $-z_i+B_i\in q_i$  
so pick $u\in S_n$ such that $\nu(u)\in G$ for each $\nu\in F$
and $M\left(\begin{array}{c}\mu_{\sigma(1)}(u)\\ \vdots\\
\mu_{\sigma(m)}(u)\end{array}\right)=\vec y\in\bigtimes_{i=1}^k(-z_i+B_i)$.

Given $\nu\in F$, $\nu(wu)=\nu(w)\nu(u)\in s^{-1}D$ so $\nu(swu)=s\nu(wu)\in D$.
Finally
$M\left(\begin{array}{c}\mu_{\sigma(1)}(swu)\\ \vdots\\
\mu_{\sigma(m)}(swu)\end{array}\right)=M\left(\begin{array}{c}\mu_{\sigma(1)}(wu)\\ \vdots\\
\mu_{\sigma(m)}(wu)\end{array}\right)=M\left(\begin{array}{c}\mu_{\sigma(1)}(w)+\mu_{\sigma(1)}(u)\\ \vdots\\
\mu_{\sigma(m)}(w)+\mu_{\sigma(m)}(u)\end{array}\right)=\vec z+\vec y\in\bigtimes_{i=1}^kB_i$.
\end{proof}

\begin{theorem}\label{PMtriag} Let $m,n\in\ben$ with $m\leq n$. Let
$M$ be an $m\times m$ lower triangular matrix with rational entries. Assume
that the entries on the diagonal are positive and the entries
below the diagonal are negative or zero. Let $F$ be a finite
nonempty set of $S_0$-preserving homomorphisms from $S_n$ to $S_0$. Let $\sigma$ be an injection
from $\nhat{m}$ to $\nhat{n}$. 
Then $P_{M\!,F\!,\,\sigma}$ is a compact subsemigroup of $(\beta\ben)^m$ containing
the idempotents of $(\beta\ben)^m$.\end{theorem}

\begin{proof} Let $k=m$. By Corollary \ref{lowertriag}, $P_{M\!,F\!,\,\sigma}$ contains the idempotents
of $(\beta\ben)^k$ so in particular $P_{M\!,F\!,\,\sigma}\neq \emptyset$. The result now follows by Lemma \ref{PMsemi}.\end{proof}

\begin{theorem}\label{PMIPR} Let $k, m,n\in\ben$ with $m\leq n$. Let
$M$ be a $k\times m$ matrix with rational entries which is image
partition regular over $\ben$. Let $F$ be a finite
nonempty set of $S_0$-preserving homomorphisms from $S_n$ to $S_0$.
Let $\sigma$ be an injection
from $\nhat{m}$ to $\nhat{n}$.
Then $P_{M\!,F\!,\,\sigma}$ is a compact subsemigroup of $(\beta \ben)^k$ containing
$\{(p,p,\ldots,p)\in(\beta\ben)^k:p$ is a minimal idempotent of $\beta\ben\}$.
\end{theorem}

\begin{proof} By Corollary \ref{IPR}, $P_{M\!,F\!,\,\sigma}$ contains
$\{(p,p,\ldots,p)\in(\beta\ben)^k:p$ is a minimal idempotent of $\beta\ben\}$ so
Lemma \ref{PMsemi} applies.\end{proof}

If $M=\left(\begin{array}{cc} 1&1\\ 1&2\end{array}\right)$ and $F$ is a finite
nonempty set of $S_0$-preserving homomorphisms from $S_n$ to $S_0$, then
by Theorem \ref{PMIPR} we have that $P_{M\!,F\!,\,\sigma}$ contains $\{(p,p):p$ is a minimal idempotent of $\beta\ben\}$
but by Theorem \ref{centralnogood}, $P_{M\!,F\!,\,\sigma}$ does not contain $\{(p_1,p_2): p_1$ and $p_2$
are minimal idempotents of $\beta\ben\}$.

Given a finite coloring of a semigroup, at least one of the color classes must
be piecewise syndetic, so results concluding that piecewise syndetic sets have
a certain property guarantee the corresponding conclusion for finite colorings.
We see now a situation where the conclusions are equivalent -- a fact that has
interesting consequences for both versions.

\begin{theorem}\label{psequiv} Let $n\in\ben$, let $\tau$ be an $S_0$-independent
homomorphism from $S_n$ to $\ben$, and let $B\subseteq\ben$. The following
statements are equivalent.
\begin{itemize}
\item[(a)] Whenever $S_0$ is finitely colored, there exists $w\in S_n$
such that $\{w(\vec x):x\in\bea^n\}$ is monochromatic and $\tau(w)\in B$.
\item[(b)] Whenever $D$ is a piecewise syndetic subset of $S_0$, there
exists $w\in S_n$ such that $\{w(\vec x):\vec x\in\bea^n\}\subseteq D$ and
$\tau(w)\in B$.
\end{itemize}\end{theorem}

\begin{proof} It is trivial that (b) implies (a), so assume that (a) 
holds and let $D$ be a piecewise syndetic subset of $S_0$.  Note that for each $r\in\ben$, there is some $m\in\ben$
such that whenever the length $m$ words in $S_0$ are $r$-colored, there
is some $w\in S_n$ of length $m$ such that $\{w(\vec x):\vec x\in \bea^n\}$ is monochromatic and $\tau(w)\in B$.
(If there is a bad $r$-coloring $\varphi_m$ of the length $m$ words for each
$m$ then $\bigcup_{m=1}^\infty\varphi_m$ is a bad $r$-coloring of $S_0$.)

Since $D$ is piecewise syndetic, pick finite nonempty $G\subseteq S_0$ such that
for every finite nonempty subset $H$ of $S_0$ there exists $s\in S_0$ with
$Hs\subseteq\bigcup_{t\in G}t^{-1}D$. Let $r=|G|$ and pick $m\in\ben$
such that whenever the length $m$ words in $S_0$ are $r$-colored, there
is some $w\in S_n$ such that $\{w(\vec x):\vec x\in\bea^n\}$ is monochromatic and $\tau(w)\in B$.
Let $H$ be the set of length $m$ words in $S_0$ and pick
$s\in S_0$ such that $Hs\subseteq\bigcup_{t\in G}t^{-1}D$. For
$u\in H$ pick $\varphi(u)\in G$ such that $us\in \varphi(u)^{-1}G$.
Pick $w\in S_n$ of length $m$ and $t\in G$ such that 
for all $\vec x\in\bea^n$, $\varphi\big(w(\vec x)\big)=t$ and $\tau(w)\in B$.
Let $w'=tws$. Then for $\vec x\in\bea^n$, $w'(\vec x)=t\big(w(\vec x)\big)s\in D$
and $\tau(w')=\tau(w)\in B$.\end{proof}

If $n=1$, the following corollary yields the statement in the second
paragraph of the abstract.

\begin{corollary}\label{coloringsemigp} Let
$n\in\ben$, let $\tau$ be an $S_0$-independent  homomorphism from $S_n$ onto $\ben$,
and let $Q=\{p\in\beta\ben:$ whenever $S_0$ is finitely colored and
$B\in p$, there exists $w\in S_n$ such that $\{w(\vec x):\vec x\in\bea^n\}$
is monochromatic and $\tau(w)\in B\}$. Then $Q$ is a compact subsemigroup 
of $\beta\ben$ containing all of the idempotents.\end{corollary}

\begin{proof} Let $k=m=1$, let $M=(1)$, and let
$F=\{h_{\vec x}:\vec x\in\bea^n\}$. By Corollary \ref{surj3}, pick
$\sigma(1)\in\nhat{n}$ such that $\tau=\mu_{\sigma(1)}$. By Theorem \ref{PMtriag},
$P_{M\!,F\!,\,\sigma}$ is a compact subsemigroup 
of $\beta\ben$ containing all of the idempotents and by 
Theorem \ref{psequiv}, $Q=P_{M\!,F\!,\,\sigma}$.\end{proof}

Recall that a set of sets ${\mathcal B}$ is said to be partition regular if whenever ${\mathcal F}$ is a finite set of sets
and $\bigcup {\mathcal F}\in {\mathcal B}$, there exist $A\in {\mathcal F}$ and $B\in {\mathcal B}$ such that $B\subseteq A$.

\begin{corollary}\label{pspr} Let $n\in\ben$ and let $\tau$
be an $S_0$-independent  homomorphism from $S_n$ to $\ben$.
Let ${\mathcal B}=\{B\subseteq\ben:$ whenever $D$ is a piecewise syndetic subset of $S_0$, there
exists $w\in S_n$ such that $\{w(\vec x):\vec x\in\bea^n\}\subseteq D$ and
$\tau(w)\in B\}$. Then ${\mathcal B}$ is partition regular.
\end{corollary}

\begin{proof} 
By Theorem \ref{psequiv},
${\mathcal B}=\{B\subseteq\ben:$ whenever $S_0$ is finitely colored, there exists $w\in S_n$
such that $\{w(\vec x):x\in\bea^n\}$ is monochromatic and $\tau(w)\in B\}$.
It is routine to show that if $k\in\ben$, $B_i\subseteq \ben$ for each $i\in\nhat{k}$,
and $\bigcup_{i=1}^kB_i$ has the property that 
whenever $S_0$ is finitely colored, there exists $w\in S_n$
such that $\{w(\vec x):x\in\bea^n\}$ is monochromatic and $\tau(w)\in \bigcup_{i=1}^k B_i$,
then some $B_i\in{\mathcal B}$.\end{proof}

Since the intersection of any collection of compact semigroups having the finite intersection 
property is a compact semigroup, it follows that there exists a smallest compact subsemigroup of $(\beta\ben)^k$
containing the idempotents of $(\beta\ben)^k$.

\begin{question}\label{qsmall} Let $k\in\ben$, let $M$ be the
$k\times k$ identity matrix, and let $\sigma$ be the identity function
on $\nhat{k}$. 
\begin{itemize}
\item[(a)] If $F=\{h_{\vec x}:\vec x\in\bea^k\}$, is $P_{M\!,F\!,\,\sigma}$ the smallest compact subsemigroup
of $(\beta\ben)^k$ containing the idempotents of $(\beta\ben)^k$?
\item[(b)] If not, does there exist a finite nonempty set $F$ of 
$S_0$-preserving homomorphisms such that
$P_{M\!,F\!,\,\sigma}$ is the smallest compact subsemigroup
of $(\beta\ben)^k$ containing the idempotents of $(\beta\ben)^k$?
\end{itemize}\end{question}

\begin{question}\label{qtriag} Let $k\in\ben$ and let $M$ and $N$ be 
$k\times k$ lower triangular matrices with positive diagonal entries
and nonpositive entries below the diagonal. 
Do there exist a finite nonempty set $F$ of 
$S_0$-preserving homomorphisms from $S_k$ to $S_0$ and a permutation
$\sigma$ of $\nhat{k}$ such that $P_{M\!,F\!,\,\sigma}\neq P_{N\!,F\!,\,\sigma}$?
\end{question}

Because of Question \ref{qsmall}, we are interested in 
the smallest compact subsemigroup
of $(\beta\ben)^k$ containing the idempotents of $(\beta\ben)^k$.

Given a compact right topological semigroup $T$,
we let $E(T)$ be the set of idempotents in $T$.  If $I$ is a set and
for each $i\in I$, $T_i$ is a compact  right topological 
semigroup, then $E(\bigtimes_{i\in I}T_i)=\bigtimes_{i\in I}E(T_i)$ because the operation in
$\bigtimes_{i\in I}T_i$ is coordinatewise.
Also by \cite[Theorem 2.23]{HS} $K(\bigtimes_{i\in I}T_i)=\bigtimes_{i\in I}K(T_i)$
so that $E\big(K(\bigtimes_{i\in I}T_i)\big)=\bigtimes_{i\in I}E\big(K(T_i)\big)$.

\begin{definition}\label{defJ} Let $T$ be a compact  right topological
semigroup and let $A\subseteq T$. Then $J_T(A)$ is the smallest compact 
subsemigroup of $T$ containing $A$.\end{definition}

We next show that $J_{(\beta\ben)^k}\big(E((\beta\ben)^k)\big)=\big(J_{\beta\ben}(E(\beta\ben))\big)^k$
for $k\in\ben$ and that a similar result applies to the minimal idempotents.  
Notice that in general 
$J_{T_1\times T_2}(A_1\times A_2) \subseteq J_{T_1}(A_1)\times J_{T_2}(A_2)$. But 
equality need not always hold even in the case that $T_1=T_2$ and $A_1=A_2$. For example,
let $A^+$ be the free semigroup on the alphabet $A=\{a,b\}$, and
$T=\beta A^+$. Then, identifying the letters of $A$ with the length one words so that 
$A$ is a subset of $T$, we have $J_T(A)\times J_T(A)=T\times T$
while $J_{T\times T}(A\times A)=\cl_{T\times T}\{(u,w)\in A^+\times A^+:|u|=|w|\}$.

\begin{theorem}\label{thJab} Let $T_1$ and $T_2$ be compact  right topological
semigroups and for $i\in\{1,2\}$ let $A_i$ be a nonempty subset of $T_i$ with $A_i\subseteq \{ab:a,b\in A_i\}$.  Then
$J_{T_1\times T_2}(A_1\times A_2)=J_{T_1}(A_1)\times J_{T_2}(A_2)$.\end{theorem}

\begin{proof}As $\Jo\times\Jt$ is a compact subsemigroup of $T_1\times T_2$ containing $A_1\times A_2$ we have immediately that $\Jot \subseteq \Jo\times \Jt$. So it remains to show that $\Jo\times\Jt  \subseteq \Jot$. Let $Y=\{q\in \Jt:(p,q)\in\Jot\,\mbox{for all}\, p\in A_1\}$.
Then $Y$ is compact and $A_2\subseteq Y$.  Further, let
$q_1,q_2\in Y$ and $p\in A_1$, and write $p=p_1p_2$ with $p_1,p_2\in A_1$. 
Then $(p_1,q_1),(p_2,q_2)\in \Jot$ and hence  $(p_1,q_1)(p_2,q_2)=(p,q_1q_2)\in \Jot$. Thus
$Y$ is a compact subsemigroup of $\Jt$ containing $A_2$ so
$Y=\Jt$.

Now let $X=\{x\in \Jo:\{x\}\times \Jt\subseteq \Jot\}$.
Then $X$ is compact and if $p\in A_1$, then $\{p\}\times\Jt=\{p\}\times Y\subseteq \Jot$,
so $A_1\subseteq X$. We next claim that  $X$ is a semigroup. In fact, let $x_1,x_2\in X$ and set
$Z=\{z\in \Jt:(x_1x_2,z)\in \Jot\}$. Then $Z$ is compact. Let $q\in A_2$ and write $q=q_1q_2$ with $q_1,q_2\in A_2$.
Then $(x_1,q_1),(x_2,q_2)\in \Jot$ and hence $(x_1x_2,q)\in \Jot$. Thus $Z$ contains $A_2$. Finally, let $z_1,z_2\in Z$.
Then since $z_1,z_2\in \Jt$ and $x_1,x_2\in X$ we deduce that $(x_1,z_1),(x_2,z_2) \in \Jot$ implying 
that $(x_1x_2,z_1z_2)\in \Jot$ and hence $z_1z_2 \in Z$. Thus $Z$ is a compact subsemigroup of 
$\Jt$ containing $A_2$ and hence $Z=\Jt$ from which it follows that $x_1x_2\in X$.  Having shown 
that $X$ is compact subsemigroup of $\Jo$ containing $A_1$ we deduce that $X=\Jo$. In conclusion, 
$\Jo\times\Jt =X\times \Jt \subseteq \Jot$ as required.\end{proof}

Notice in particular that if for $i\in\{1,2\}$, $A_i$ is a nonempty subset of $E(T_i)$, then
$A_i\subseteq \{ab:a,b\in A_i\}$, so $J_{T_1\times T_2}(A_1\times A_2)=J_{T_1}(A_1)\times J_{T_2}(A_2)$.

\begin{corollary}\label{corJE} Let $k\in\ben$. The smallest
compact subsemigroup of $(\beta\ben)^k$ containing
the idempotents of $(\beta\ben)^k$ is $\big(J_{\beta\ben}\big(E(\beta\ben)\big)\big)^k$.
The smallest
compact subsemigroup of $(\beta\ben)^k$ containing the minimal 
idempotents of $(\beta\ben)^k$ is $\big(J_{\beta\ben}\big(E(K(\beta\ben))\big)\big)^k$.
\end{corollary}

\begin{proof} By Theorem \ref{thJab} and induction, 
$$J_{(\beta\ben)^k}((E(\beta\ben))^k)=\big(J_{\beta\ben}\big(E(K(\beta\ben))\big)\big)^k$$
and we already observed that the set of idempotents of 
$(\beta\ben)^k$ is $\big(E(\beta\ben)\big)^k$. The second conclusion is
essentially the same. \end{proof}

We note now that the version of Theorem \ref{thJab} for infinite products
is also valid.

\begin{theorem}\label{thJinf} Let $I$ be an infinite set. For each
$i\in I$, let $T_i$ be a compact  right topological semigroup 
and let $A_i$ be a nonempty subset of $T_i$ such that $A_i\subseteq\{ab:a,b\in A_i\}$. Then
$J_{\times_{i \in I}T_i}(\bigtimes_{i\in I}A_i)=\bigtimes_{i\in I}J_{T_i}(A_i)$.
\end{theorem}

\begin{proof} Let $Y=\bigtimes_{i\in I}T_i$. For each $i\in I$, choose $e_i\in A_i$.
Given $F\in\pf(I)$,
let $Y_F=\bigtimes_{i\in F}T_i$, let $Z_F=\bigtimes_{i\in I\setminus F}T_i$,
let $$X_F=\{\vec x\in\bigtimes_{i\in I}J_{T_i}(A_i):(\forall i\in I\setminus F)(x_i=e_i)\}\,,$$
and let $B_F=\{\vec x\in\bigtimes_{i\in I}A_i:(\forall i\in I\setminus F)(x_i=e_i)\}$.

We shall show that for each $F\in\pf(I)$, $X_F\subseteq J_Y(\bigtimes_{i\in I}A_i)$.
Let $F\in\pf(I)$ be given.  Now $X_F$ is topologically and algebraically
isomorphic to $$\bigtimes_{i\in F}J_{T_i}(A_i)\times\bigtimes_{i\in I\setminus F}\{e_i\}\,,$$
$B_F$ is topologically and algebraically
isomorphic to $\bigtimes_{i\in F}A_i\times\bigtimes_{i\in I\setminus F}\{e_i\}$,
and $\bigtimes_{i\in I\setminus F}\{e_i\}\subseteq J_{Z_F}(\bigtimes_{i\in I\setminus F}\{e_i\})$.
So using Theorem \ref{thJab}  we have
$$\begin{array}{rl} X_F&\approx \bigtimes_{i\in F}J_{T_i}(A_i)\times\bigtimes_{i\in I\setminus F}\{e_i\}\\
&\subseteq J_{Y_F}(\bigtimes_{i\in F}A_i)\times J_{Z_F}(\bigtimes_{i\in I\setminus F}\{e_i\})\\
&=J_{Y_F\times Z_F}(\bigtimes_{i\in F}A_i\times \bigtimes_{i\in I\setminus F}\{e_i\})\\
&\approx J_Y(B_F)\\
&\subseteq J_Y(\bigtimes_{i\in I}A_i)\,.\end{array}$$

Next we claim that $\bigtimes_{i\in I}J_{T_i}(A_i)
\subseteq \cl_Y \bigcup_{F\in\pf(I)}X_F$.  To see this, let
$\vec z\in\bigtimes_{i\in I}J_{T_i}(A_i)$ and let
$U$ be a neighborhood of $\vec z$ in $Y$. Pick $F\in \pf(I)$
and for each $i\in F$, pick a neighborhood $V_i$ of $z_i$
in $T_i$ such that $\bigcap_{i\in F}\pi_i^{-1}[V_i]\subseteq U$.
Define $\vec x\in Y$ by $x_i=\left\{\begin{array}{cl} z_i&\hbox{if }i\in F\\
e_i&\hbox{if }i\in I\setminus F\,.\end{array}\right.$
Then $\vec x\in U\cap X_F$.  Therefore 
$\bigtimes_{i\in I}J_{T_i}(A_i)\subseteq J_Y(\bigtimes_{i\in I}A_i)$. 
Since $\bigtimes_{i\in I}J_{T_i}(A_i)$ is a compact semigroup containing
$\bigtimes_{i\in I}A_i$, the reverse inclusion is immediate.\end{proof}

The curious reader may wonder what the situation is with respect to
the smallest semigroup containing a given set. Given a semigroup
$T$ and a nonempty subset $A$ of $T$, let $J'_T(A)$ be the smallest 
subsemigroup of $T$ containing $A$, that is the set of all finite
products of members of $A$ in any order allowing repetition. If $T_1$ and
$T_2$ are any semigroups and $A_1$ and $A_2$ are nonempty
subsets of $T_1$ and $T_2$ respectively such that 
$A_i\subseteq\{ab:a,b\in A_i\}$ for $i\in\{1,2\}$, then
$J'_{T_1\times T_2}(A_1\times A_2)=J'_{T_1}(A_1)\times J'_{T_2}(A_2)$.
This follows from the proof of Theorem \ref{thJab} by
omitting all references to the topology.

However, the analogue of Theorem \ref{thJinf} need not hold.  To
see this, let $T$ be the set of words over the 
alphabet $\{a_n:n\in\ben\}$ that have no adjacent occurrences of
any one letter. Given $u,w\in T$, then $u\cdot w$ is ordinary
concatenation unless $u=xa_n$ and $w=a_ny$ for some 
$n\in\ben$ and some $x,y\in T\cup\{\emp\}$, in which
case $u\cdot w=xa_ny$.  Let $A$ be the set of idempotents in $T$, that
is $A$ is the set of length one words.  Then $J'_T(A)=T$ but
$\{\vec x\in\bigtimes_{n=1}^\infty T:\{|x_n|:n\in\ben\}$ is bounded$\}$
is a proper subsemigroup of $\bigtimes_{n=1}^\infty T$ containing the idempotents.

\section{Compact ideals of $(\beta S)^k$}\label{seccptideals}

In this section we deal with results related to the Hales-Jewett 
Theorem and its extensions by themselves.  The first 
result here is motivated by the following theorem characterizing
image partition regular matrices.

\begin{theorem}\label{charIPR} Let $k,m\in\ben$ and let
$M$ be a $k\times m$ matrix with entries from $\beq$.  The
following statements are equivalent.
\begin{itemize}
\item[(a)] $M$ is image partition regular over $\ben$.
\item[(b)] For every central subset $D$ of $\ben$, there
exists $\vec x\in\ben^m$ such that $M\vec x\in D^k$.
\item[(c)] For every central subset $D$ of $\ben$, 
$\{\vec x\in\ben^m:M\vec x\in D^k\}$ is central in $\ben^m$.
\end{itemize}\end{theorem}

\begin{proof} These are statements (a), (h), and (i) of
\cite[Theorem 15.24]{HS}.\end{proof}

As a consequence of Theorem \ref{matrix} (with $k=m=1$, $T=\ben$, $M=(1)$,
$\tau=\mu_1$, $F=\{h_{\vec x}:\vec x\in\A^n\}$, and
$p$ any idempotent in $\beta\ben$) we have that whenever
$D$ is piecewise syndetic in $S_0$ and $n\in\ben$, 
there exists $w\in S_n$ such that $\{w(\vec x):\vec x\in \A^n\}\subseteq D$.
Since whenever $S_0$ is finitely colored, one color class must
be piecewise syndetic, we see that Theorem \ref{ndimHJ} follows.
And, since central sets are piecewise syndetic, 
we have that whenever
$D$ is central in $S_0$ and $n\in\ben$, 
there exists $w\in S_n$ such that $\{w(\vec x):\vec x\in \A^n\}\subseteq D$.

\begin{theorem}\label{centralHJ} Let $n\in\ben$ and let $D$ be a central subset of $S_0$.
Let $F$ be a finite nonempty set of $S_0$-preserving homomorphisms from $S_n$ into $S_0$.
Then $\{w\in S_n:(\forall \nu\in F)(\nu(w)\in D)\}$ is central in $S_n$.
\end{theorem}

\begin{proof} Let $T=S_n\cup S_0$ and extend each $\nu\in F$ to all of $T$ by
defining $\nu$ to be the identity on $S_0$. By Theorem \ref{blass}(2), pick a
central subset $Q$ of $T$ such that for each $t\in Q$, $\{\nu(t):\nu\in F\}\subseteq D\}$.
Since $S_n$ is an ideal of $T$, $Q\cap S_n$ is central in $S_n$ and
$Q\cap S_n\subseteq\{w\in S_n:(\forall \nu\in F)(\nu(w)\in D)\}$.
\end{proof} 

We  conclude this section by investigating ideals related to the
extensions of the Hales-Jewett Theorem.

\begin{definition}\label{defRn} For $n\in\ben$, 
$$R_n=\{p\in\beta S_0:(\forall B\in p)(\exists w\in S_n)(\{w(\vec x):\vec x\in \A^n\}\subseteq B)\}\,.$$
\end{definition}

There are numerous ways to use known results to show that each $R_n\neq\emp$.
From the point of view of this paper, it is probably easiest to invoke
Theorem \ref{matrix} as discussed above.

\begin{theorem}\label{Rnideal} Let $n\in\ben$. Then $R_n$ is a compact two sided ideal of
$\beta S_0$.\end{theorem}

\begin{proof} We have that $R_n\neq\emp$ and it is trivially compact.  Let $p\in R_n$ and
let $q\in\beta S_0$.  To see that $R_n$ is a left ideal, let $B\in qp$.
Pick $u\in S_0$ such that $u^{-1}B\in p$ and pick $w\in S_n$ such that
$\{w(\vec x):\vec x\in \A^n\}\subseteq u^{-1}B$. Then $uw\in S_n$
and $\{(uw)(\vec x):\vec x\in \A^n\}\subseteq B$.

To see that $R_n$ is a right ideal, let $B\in pq$.
Pick $w\in S_n$ such that $\{w(\vec x):\vec x\in \A^n\}\subseteq\{u\in S_0:u^{-1}B\in q\}$.
Pick $u\in\bigcap_{\vec x\in \A^n}w(\vec x)^{-1}B$.
Then $wu\in S_n$ and $\{(wu)(\vec x):\vec x\in \A^n\}\subseteq B$. \end{proof}

\begin{theorem}\label{Rndecrease} Let $n\in\ben$. Then $R_{n+1}\subseteq R_n$.\end{theorem}

\begin{proof} Let $p\in R_{n+1}$ and let $B\in p$. Pick
$w\in S_{n+1}$ such that $\{w(\vec x):\vec x\in \A^{n+1}\}\subseteq B$.
Define $u\in S_n$ by $u=w(v_1,v_2,\ldots,v_n,v_n)$. Then given
$\vec x\in \A^n$, $u(\vec x)=w(x_1,x_2,\ldots,x_n,x_n)\in B$. \end{proof}

\begin{lemma}\label{HJbig} For each $r,n\in\ben$ there exists $m\in\ben$ such 
that for all $k\geq m$, if $\varwordspl k0$ is 
$r$-colored, there exists $w\in \varwordspl kn$ such that 
$\{w(\vec x):\vec x\in \A^n\}$ is monochromatic.\end{lemma}

\begin{proof} Let $r,n\in \ben$. By Theorem \ref{ndimHJ}, whenever
$S_0$ is $r$-colored, there exists $w\in S_n$ such that $\{w(\vec x):
\vec x\in\bea^n\}$ is monochromatic. As in the proof of Theorem \ref{psequiv}, pick $m\in\ben$ such that 
whenever $\varwordspl m0$ is $r$-colored, there exists $w\in \varwordspl mn$ such that
$\{w(\vec x):\vec x\in \A^n\}$ is monochromatic. Let $k>m$ and pick
$c\in \A$. Let $\varphi:\varwordspl k0\to\nhat{r}$ and define $\psi:\varwordspl m0\to\nhat{r}$
by $\psi(u)=\varphi(uc^{k-m})$. Pick $w\in \varwordspl mn$ such
that $\psi$ is constant on $\{w(\vec x):\vec x\in \A^n\}$. Define
$u\in \varwordspl kn$ by $u=wc^{k-m}$. Then $\varphi$ is constant 
on $\{u(\vec x):\vec x\in \A^n\}$.\end{proof}

\begin{theorem}\label{clKRn} $\cl K(\beta S_0)\subsetneq \bigcap_{n=1}^\infty R_n$.
\end{theorem}

\begin{proof} That $\cl K(\beta S_0)\subseteq \bigcap_{n=1}^\infty R_n$ is
an immediate consequence of Theorem \ref{Rnideal}.

Let $B=\bigcup_{k=1}^\infty \varwordspl {k!}0$.  We claim first that $B$ 
is not piecewise syndetic, so that $\overline B\cap \cl K(\beta S_0)=\emp$.
We need to show that there is no
$G\in\pf(S_0)$ such that for all $F\in\pf(S_0)$
there exists $x\in S_0$ such that $Fx\subseteq \bigcup_{t\in G}t^{-1}B$.
Suppose we have such $G$ and let $m=\max\{|t|\,:\,t\in G\}$, let $r=m!$, 
pick $b\in \A$, and let $F=\{b^r,b^{2r}\}$.
Pick $t,s\in G$ and $x\in S_0$ such that $tb^rx\in B$ and $sb^{2r}x\in B$.
Then $|tb^rx|=n!$ for some $n>m$ and $|sb^{2r}x|=k!$ for some $k$.
Now $k!=|sb^{2r}x|=|sb^r|+n!-|t|>n!$ so $k!\geq (n+1)!$ so
$|sb^r|+n!-|t|\geq (n+1)!$. Thus $n\cdot n!=(n+1)!-n!<|sb^r|=|s|+r\leq m+m!<n+n!$,
a contradiction.

Now let $n\in\ben$. We will show that $\overline B\cap R_n\neq\emp$. 
Let ${\cal R}=\{D\subseteq S_0:$ whenever $D$ is finitely colored,
there exists $w\in S_n$ such that $\{w(\vec x):\vec x\in \A^n\}$ is
monochromatic$\}$. Notice that ${\mathcal R}$ is partition regular. It suffices to show that $B\in {\cal R}$, for then
by \cite[Theorem 3.11]{HS} there exists $p\in\beta S_0$ such that
$B\in p$ and $p\subseteq{\cal R}$ so that $p\in\overline B\cap R_n$.
So let $r\in\ben$ and let $\varphi:B\to\nhat{r}$. Pick $m$ as guaranteed
by Lemma \ref{HJbig} for $r$ and $n$. The $\varphi$ is an $r$-coloring of $\varwordspl {m!}0$
so pick $w\in \varwordspl{m!}{n}$ such that 
$\{w(\vec x):\vec x\in \A^n\}$ is monochromatic.

Since $\{\overline B\cap R_n:n\in\ben\}$ is a collection of closed
sets with the finite intersection property, we have that
$\overline B \cap \bigcap_{n=1}^\infty R_n\neq\emp$.\end{proof}

\begin{theorem}[Deuber, Pr\"omel, Rothschild, and Voigt]\label{thDPRV}
Let $n,r\in\ben$. There exist $m\in\ben$ and $C_{n,r}\subseteq \varwordspl m0$ such that
\begin{itemize}
\item[(1)] there does not exist $w\in \varwordspl{m}{n+1}$ with
$\{w(\vec x):\vec x\in \A^{n+1}\}\subseteq C_{n,r}$ and
\item[(2)] whenever $C_{n,r}$ is $r$-colored, there exists
$w\in \varwordspl{m}{n}$ such that $\{w(\vec x):\vec x\in \A^n\}$ is
monochromatic.
\end{itemize}\end{theorem}

\begin{proof} This is the ``main theorem" of \cite{DPRV}.\end{proof}

\begin{theorem}\label{Rnproper} Let $n\in\ben$. Then $R_{n+1}\subsetneq R_n$.
\end{theorem}

\begin{proof} For each $r\in\ben$ pick $m(r)$ and $C_{n,r}$ as guaranteed for $r$ and $n$ by
Theorem \ref{thDPRV}. Choose an increasing sequence $\langle r_i\rangle_{i=1}^\infty$ such that
the sequence $\langle m(r_i)\rangle_{i=1}^\infty$ is strictly increasing
and let $D_i=C_{n,r_i}$ for each $i$.  Let $E=\bigcup_{i=1}^\infty D_i$.
There does not exist $w\in S_{n+1}$ such that $\{w(\vec x):\vec x\in \A^{n+1}\}
\subseteq E$ because any such $w$ would have to have length $m(r_i)$ for some
$i$, and then one would have $\{w(\vec x):\vec x\in \A^{n+1}\}\subseteq C_{n,r_i}$.
Thus $\overline{E}\cap R_{n+1}=\emp$.

As in the proof of Theorem \ref{clKRn}, let
${\cal R}=\{D\subseteq S_0:$ whenever $D$ is finitely colored,
there exists $w\in S_n$ such that $\{w(\vec x):\vec x\in \A^n\}$ is
monochromatic$\}$. It suffices to show that $E\in{\cal R}$ so let
$k\in\ben$ and let $\varphi:E\to \nhat{k}$.  Pick $i$ such that
$r_i\geq k$. Then $\varphi_{|D_i}:D_i\to\nhat{r_i}$  so
pick $w\in \varwordspl{m(r_i)}{n}$ such that $\varphi$ is constant
on $\{w(\vec x):\vec x\in \A^n\}$.\end{proof}

\section{A simpler proof of an infinitary extension}\label{infext}

We set out in this section to provide a proof of \cite[Theorem 2.12]{CHS}
applied to the simpler description of $n$-variable words which we have
been using.  As defined in this paper, what is called the set of $n$-variable
words in \cite{CHS}, is what we call the set of strong $n$-variable words
where we take $D=E=\{e\}$ in \cite{CHS}, take the function $T_e$ to
be the identity, and let $v_n=(e,\nu_n)$ for $n\in\ben$. As we remarked earlier, in \cite[Theorem 5.1]{CHS} it was
shown that the version of the Graham-Rothschild that we are using here is sufficient 
to derive the full original version as used in \cite{CHS} and  \cite{GR}. Using
that simplified notion, Corollary \ref{infGR} implies \cite[Theorem 2.12]{CHS}
and has a vastly simpler proof.

The first few results apply to an arbitrary nonempty alphabet $\A$. For the results
of this section, except for Corollary \ref{corGR}, we do not need to assume that $\A$ is finite.

\begin{definition} \label{defTn} For $n\in\ben$, $T_n$ is the free semigroup
over $\A\cup\{v_1,v_2,\ldots,v_n\}$. Also we set $T_0=S_0$ and $T=\bigcup_{i\in \omega}T_i.$ \end{definition} 

Note that for $n\in\ben$, $S_n$ is a proper subset of $T_n$ and that $T_n\subseteq T_{n+1}$.

For $u=l_1l_2\cdots l_m\in T$ with $|u|=m$ we
define $h_u:T\to T$ by stating that  $h_u(w)$ is the result of replacing each occurrence of $v_i$ in $w$ by $l_i$ for
$i\in \nhat{m}$. (Thus, if $w\in S_m$, $h_u(w)=w(u)$ as defined in Definition \ref{defnvar}.) Denote also by $h_u$ the continuous extension of $h_u$
taking $\beta T$ to $\beta T$. Observe that, if $u\in \varwords mk$, then $h_u[T_m]\subseteq T_k$.

\begin{definition} For $\alpha\in \ben\cup\{\omega\}$, a reductive sequence of height $\alpha$ over $\A$ is 
a sequence of minimal idempotents $\langle p_t\rangle_{t<\alpha}$ with $p_t\in E\big(K(\beta S_t)\big)$ 
such  that for each $i,j\in\omega$ with  $0\leq j<i<\alpha$ one has $p_i\leq p_j$ and 
$h_u(p_i)=p_j$ for each $u\in\varwords ij$. \end{definition} 

\begin{lemma}\label{Kequal} Let $i\in\omega$. Then $K(\beta S_i)=K(\beta T_i)$.
\end{lemma}

\begin{proof} 
We have that $S_i$ is an ideal of $T_i$ so by \cite[Corollary 4.18]{HS}
$\beta S_i$ is an ideal of $\beta T_i$. Therefore
$K(\beta T_i)\subseteq \beta S_i$ so that by 
\cite[Theorem 1.65]{HS} $K(\beta S_i)=K(\beta T_i)$. 
\end{proof}

\begin{lemma}\label{smaller} Let $k,m\in\omega$ with $k<m$ and
let $p\in E(\beta S_k)$. There exists $q\in E\big(K(\beta S_m)\big)$
such that $q<p$.
\end{lemma}

\begin{proof} We have $\beta S_m\cup\beta S_k\subseteq T_m$.
Pick $q\in E\big(K(\beta T_m)\big)$ such that
$q\leq p$. By Lemma \ref{Kequal}, $q\in K(\beta S_m)$ and
since $\beta S_k\cap\beta S_m=\emp$, $q\neq p$.\end{proof}

\begin{lemma}\label{varik} Let $\alpha\in\ben\cup\{\omega\}$ and let $\langle p_t\rangle_{t<\alpha}$
be a reductive sequence of height $\alpha$. 
For each $t<\alpha$, $p_t\in E\big(K(\beta \widetilde S_t)\big)$.
\end{lemma}

\begin{proof} Since each $p_t$ is an idempotent, it suffices to show that $p_t\in K(\beta  \widetilde S_t)$. 
Given $t<\alpha$, $\widetilde S_t$ is a right ideal of 
$S_t$ so $\beta \widetilde S_t$ is a right ideal of 
$\beta S_t$ and thus $K(\beta \widetilde S_t)\cap K(\beta S_t)\neq\emp$ so
that by \cite[Theorem 1.65]{HS}, $K(\beta \widetilde S_t)=\beta \widetilde S_t\cap K(\beta S_t)$.
Thus it suffices to show that each $p_t\in \beta \widetilde S_t$.
We proceed by induction on $t$. For $t=0$, we have
$p_0\in\beta S_0=\beta \widetilde S_0$. Now assume that
$t+1<\alpha$ and $p_t\in \beta \widetilde S_t$. We need to show
that $\widetilde S_{t+1}\in p_{t+1}$. 
We begin by observing that if $w\in \widetilde S_t$ then $S_{t+1}\subseteq w^{-1}\widetilde S_{t+1}$ from which it follows that $\widetilde S_t\subseteq \{w\in T_{t+1}:w^{-1}\widetilde S_{t+1} \in p_{t+1}\}.$ Now since $\widetilde S_t\in p_t$ we have that
$\{w\in T_{t+1}:w^{-1}\widetilde S_{t+1} \in p_{t+1}\}\in p_t$ or equivalently that $\widetilde S_{t+1}\in p_tp_{t+1}.$ The result now follows from the fact that $p_{t+1}\leq p_t$ and hence in particular $p_{t+1}=p_tp_{t+1}.$ \end{proof}

We now introduce some new notation. 
We fix a nonempty (possibly infinite) alphabet $\A$ together with an infinite sequence of symbols 
$\{x_1,x_2,x_3,\ldots \}$ each of which is not a member of $\A\cup\{v_i:i\in\ben\}$. We let
$\A^{(0)}=\A$ and for $m\in\ben$, we let $\A^{(m)}=\A\cup\{x_1,x_2,\ldots,x_m\}$.
For each $m\in \omega$ we let $S^{(m)}$ denote the free semigroup over $\A^{(m)}$. For 
each $i\in \ben $ we let $S^{(m)}_i$ denote the set of all $i$-variable words over the alphabet 
$\A^{(m)}$ and let $\widetilde S^{(m)}_i$ denote the set of all strong $i$-variable words over
$\A^{(m)}$.  $T^{(j)}_i$ will denote the free semigroup of all words over the alphabet $\A^{(j)}\cup \{v_1,v_2,\ldots ,v_i\}$. 
Also, for each $j\in\omega$, let $S^{(j)}_0=\widetilde S^{(j)}_0=T^{(j)}_0=S^{(j)}$, and
let $T^{(j)}=\bigcup_{i=1}^\infty T^{(j)}_i$. Then $T^{(j)}$ is the set of all words
over $\A^{(j)}\cup\{v_i:i\in\ben\}$.

To each $u=u_1u_2\cdots u_m \in T^{(j)}$ with $|u|=m$ we associate a morphism 
$h_u:\bigcup_{j\in\omega}T^{(j)}\rightarrow \bigcup_{j\in\omega}T^{(j)}$ where for each $w\in \bigcup_{j\in\omega}T^{(j)}$,  
$h_u(w)$ is obtained from $w$ by replacing each occurrence of $v_i$ in $w$ by $u_i$ for each $i\in\nhat{m}$.
We also denote by $h_u$ its continuous extension taking $\beta(\bigcup_{j\in\omega} T^{(j)})$ to $\beta(\bigcup_{j\in\omega}T^{(j)})$.
 Also, for each $i,j\in \ben$ we define the morphism, $\tau^{(j)}_i:T^{(j)}\rightarrow T^{(j-1)}$ 
where $\tau^{(j)}_i(w)$ is the word obtained from $w$ by replacing every occurrence of $x_j$ by $v_i$ 
and leaving all other symbols unchanged.
We also denote by $\tau_i^{(j)}$ the continuous extension of $\tau_i^{(j)}$ taking $\beta T^{(j)}$ to $\beta T^{(j-1)}$.
Note that $\tau_i^{(j)}[T_{i-1}^{(j)}]=T_i^{(j-1)}$ and the restriction of
$\tau_i^{(j)}$ to $T_{i-1}^{(j)}$ is an isomorphism onto $T_i^{(j-1)}$.
Consequently the restriction of
$\tau_i^{(j)}$ to $\beta T_{i-1}^{(j)}$ is an isomorphism onto $\beta T_i^{(j-1)}$.

\begin{lemma}\label{Gisgroup} Let $m\in\omega$, let $i\in\ben\setminus\{1\}$, and
assume that $p_{i-1}^{(m)}\in \beta S_{i-1}^{(m)}$ and 
$p_{i-1}^{(m+1)}\in K(\beta S_{i-1}^{(m+1)})$. Then 
\[G_i^{(m)}=p_{i-1}^{(m)}\tau^{(m+1)}_i(p^{(m+1)}_{i-1})\beta S_i^{(m)}\cap \beta S_i^{(m)}\tau_i^{(m+1)}(p^{(m+1)}_{i-1})p_{i-1}^{(m)}\]
is a group contained in $K(\beta S_i^{(m)})$.
 \end{lemma}

\begin{proof} We will show that $G_i^{(m)}$ is the intersection of a minimal right ideal and a minimal left ideal of $\beta S_i^{(m)}$ and hence by \cite[Theorem 1.61]{HS}, 
$G_i^{(m)}$ is a group contained in  $K(\beta S_i^{(m)})$. Notice first that $\beta S_{i-1}^{(m)}\cup\beta S_{i}^{(m)}\subseteq \beta T_{i}^{(m)}$ so that the products $p_{i-1}^{(m)}\tau^{(m+1)}_i(p^{(m+1)}_{i-1})$ and 
$\tau_i^{(m+1)}(p^{(m+1)}_{i-1})p_{i-1}^{(m)}$ are computed in $\beta T_{i}^{(m)}$.

Since $\tau_i^{(m+1)}$ is an isomorphism on $\beta T_{i-1}^{(m+1)}$, we have $\tau_i^{(m+1)}[K(\beta T_{i-1}^{(m+1)})]\break
=K(\beta T_{i}^{(m)})$.
Since $p_{i-1}^{(m+1)}\in K(\beta S_{i-1}^{(m+1)})$ and $K(\beta S_{i-1}^{(m+1)})=K(\beta T_{i-1}^{(m+1)})$
(Lemma \ref{Kequal} with the underlying alphabet taken to be $\A^{(m+1)}),$ 
it follows that $\tau_i^{(m+1)}(p_{i-1}^{(m+1)})\in K(\beta T_{i}^{(m)})$. Since
$p_{i-1}^{(m)}\in \beta T_{i-1}^{(m)}\subseteq \beta T_i^{(m)}$, we have that
$p_{i-1}^{(m)}\tau^{(m+1)}_i(p^{(m+1)}_{i-1})\in K(\beta T_i^{(m)})=K(\beta S_i^{(m)})$ so that
$p_{i-1}^{(m)}\tau^{(m+1)}_i(p^{(m+1)}_{i-1})\beta S_i^{(m)}$ is a minimal right ideal of
$\beta S_i^{(m)}$.  Similarly $ \beta S_i^{(m)}\tau_i^{(m+1)}(p^{(m+1)}_{i-1})p_{i-1}^{(m)}$ is a minimal
left ideal of $\beta S_i^{(m)}.$
\end{proof}

\begin{definition}\label{HRS} For $\alpha\in \ben\cup\{\omega\}$, a {\it reductive array of height $\alpha$\/}
over $\A$ is an $\alpha\times\omega$ array of minimal idempotents $\langle p_i^{(m)}\rangle_{i<\alpha}^{m<\omega}$ with
$p_i^{(m)}\in E\big(K(\beta S_i^{(m)})\big)$ satisfying the following conditions:
\begin{enumerate}
\item For each $m\in \omega$ the sequence $\langle p_i^{(m)}\rangle_{i<\alpha}$ 
is a reductive sequence of height $\alpha$ over $\A^{(m)}$.
\item $\tau_1^{(m)}(p^{(m)}_0)=p_1^{(m-1)}$ for each $m\in \ben$.
\item For each $m\in \omega$ and $i<\alpha$ with $i\geq 2$, $p_i^{(m)}$ is the identity of the group 
\[G_i^{(m)}=p_{i-1}^{(m)}\tau^{(m+1)}_i(p^{(m+1)}_{i-1})\beta S_i^{(m)}\cap \beta S_i^{(m)}\tau_i^{(m+1)}(p^{(m+1)}_{i-1})p_{i-1}^{(m)}\,.\]
\end{enumerate}
\end{definition}

\begin{definition}\label{Xij} Let $i,j,m\in \omega$, 
with $j<i$. $X_{i,j}^{(m)}$ will denote the set of words in $\varwordsj mij$ in 
which $v_j$ occurs only as the last letter. \end{definition}

\begin{lemma} \label{hutau} Let $m,n\in \ben$ and let 
$\vec p=\langle p_0,p_1,p_2,\ldots,p_n\rangle$ be a reductive sequence of height $n+1$ over
$\A^{(m)}$. Let $p_{n+1}$ be a minimal idempotent in $\beta S^{(m)}_{n+1}$ 
for which $p_{n+1}\leq p_n$.  Let $j\in \omega$ with $j\leq n$ and let $u=u_1u_2\cdots u_{n+1}\in \varwordsj m{n+1}j$.
\begin{itemize}
\item[(1)] If $u\notin X_{n+1,j}^{(m)}$, then $h_u(p_{n+1})=p_j$.
\item[(2)] If $u\in X_{n+1,j}^{(m)}$, then for $w\in T^{(m+1)}_n$, 
$h_u\big(\tau^{(m+1)}_{n+1}(w)\big)=\tau^{(m+1)}_j\big(h_u(w)\big)$.
\end{itemize}
\end{lemma}

\begin{proof} (1) Assume $u\notin X_{n+1,j}^{(m)}$. We have that
$h_u(p_{n+1})$ and $h_u(p_n)$ are both idempotents in $\beta S_j^{(m)}$ and
$h_u(p_{n+1})\leq h_u(p_n)$ because $h_u$ is a homomorphism.  Assume first that
$j<n$ and let $s=u_1u_2\cdots u_{n}$. Then $s\in\varwordsj mnj$ so $h_s(p_n)=p_j$ and
since $h_s$ and $h_u$ agree on $S_n^{(m)}$, $h_u(p_n)=p_j$. If $j=n$, then
$u=v_1v_2\cdots v_nu_{n+1}$ so $h_u$ is the identity on $S_n^{(m)}$ and again
$h_u(p_n)=p_j$.  Consequently, $h_u(p_{n+1})\leq p_j$ and $p_j$ is 
minimal in $\beta S_j^{(m)}$ so $h_u(p_{n+1})= p_j$.

(2) It suffices to show that
$h_u\big(\tau^{(m+1)}_{n+1}(l)\big)=\tau^{(m+1)}_j\big(h_u(l)\big)$ for each 
$l\in \A^{(m+1)}\cup\{v_1,v_2,\ldots,v_n\}$.
Now $h_u\big(\tau^{(m+1)}_{n+1}(x_{m+1})\big)=h_u(v_{n+1})=u_{n+1}=v_j$
and $\tau^{(m+1)}_j\big(h_u(x_{m+1})\big)=\tau^{(m+1)}_j(x_{m+1})=v_j$.
If $l\in \A^{(m)}$, then both sides leave $l$ fixed.
Finally if $i\in\nhat{n}$, then 
$h_u\big(\tau^{(m+1)}_{n+1}(v_i)\big)=h_u(v_i)=u_i$
and $\tau^{(m+1)}_j\big(h_u(v_i)\big)=\tau^{(m+1)}_{n+1}(u_i)=
u_i$ because $u_i\neq x_{m+1}$.
\end{proof}

\begin{lemma}\label{height2} Let $q\in E\big(K(\beta S_0)\big)$
and let $r\in E\big(K(\beta S_1)\big)$ such that
$r<q$. There is a reductive array $\langle p_i^{(m)}\rangle_{i<2}^{m<\omega}$
of height $2$ over $\A$ such that $p_0^{(0)}=q$ and $p_1^{(0)}=r$.\end{lemma}

\begin{proof} Let $p_0^{(0)}=q$ and $p_1^{(0)}=r$. Let $m\in\ben$
and assume that we have chosen $\langle p_i^{(t)}\rangle_{i<2}^{t<m}$
such that for each $t<m$ and each $i\in\{0,1\}$, $p_i^{(t)}\in E\big(K(\beta S_i^{(t)})\big)$
and $p_1^{(t)}<p_0^{(t)}$.
By Lemma \ref{Kequal}, $p_1^{(m-1)}\in K(\beta T_1^{(m-1)})$.
Since $\tau_1^{(m)}$ is an isomorphism from $\beta T_0^{(m)}$ onto
$\beta T_1^{(m-1)}$, we may let $p_0^{(m)}$ be the unique
member of $E\big(K(\beta T_0^{(m)})\big)$ such that
$\tau_1^{(m)}(p_0^{(m)})=p_1^{(m-1)}$. By Lemma \ref{smaller}, we
may pick $p_1^{(m)}\in E\big(K(\beta S_1^{(m)})\big)$ such that
$p_1^{(m)}<p_0^{(m)}$.

We need to show that for each $u\in\varwordsj m10$, 
$h_u(p_1^{(m)})=p_0^{(m)}$ so let $u\in\varwordsj m10$.
Now $p_0^{(m)}$ and $p_1^{(m)}$ are in $\beta T_1^{(m)}$,
$p_1^{(m)}\leq p_0^{(m)}$, and $h_u$ is a homomorphism
so $h_u(p_1^{(m)})\leq h_u(p_0^{(m)})$. Since $h_u$ is the
identity on $\beta T_0^{(m)}$,
$h_u(p_0^{(m)})=p_0^{(m)}$ so
$h_u(p_1^{(m)})\leq p_0^{(m)}$.
Since $p_0^{(m)}$ is minimal in $\beta T_0^{(m)}$,
$h_u(p_1^{(m)})= p_0^{(m)}$ as required.\end{proof}

\begin{theorem}\label{hyperextend} Let $n\in\ben$ and assume that
$\langle p_i^{(m)}\rangle_{i<n+1}^{m<\omega}$ is a reductive array
of height $n+1$ over $\A$. There exist unique $p_{n+1}^{(m)}$ for each $m<\omega$
such that $\langle p_i^{(m)}\rangle_{i<n+2}^{m<\omega}$ is a reductive array
of height $n+2$ over $\A$.
\end{theorem} 

\begin{proof} For each $m<\omega$, let $p_{n+1}^{(m)}$ be the identity
of the group $G_{n+1}^{(m)}$. This is required by Definition \ref{HRS}(3), 
so the uniqueness is satisfied. Let $m<\omega$ be given. We need to show that
$\langle p_i^{(m)}\rangle_{i<n+2}$ is a reductive sequence over $\A^{(m)}$.
Since $G_{n+1}^{(m)}\subseteq p_n^{(m)}\beta S_{n+1}^{(m)}\cap \beta S_{n+1}^{(m)}p_n^{(m)}$ we have
that $p_{n+1}^{(m)}\leq p_n^{(m)}$. And by Lemma \ref{Gisgroup} we have
that $p_{n+1}^{(m)}\in E\big(K(\beta S_{n+1}^{(m)})\big)$.

Now let $0\leq j<i<n+2$ and let $u\in\varwordsj mij$. We need to show that
$h_u(p_{i}^{(m)})=p_j^{(m)}$. If $i<n+1$, this holds by assumption, so
assume that $i=n+1$ so that $u\in\varwordsj m{n+1}j$.
 If $j=0$, then $h_u$ is the identity
on $\beta S_0^{(m)}$ so $h_u(p_{n+1}^{(m)})\leq h_u(p_0^{(m)})=p_0^{(m)}$
and $h_u(p_{n+1}^{(m)})\in\beta S_0^{(m)}$ so $h_u(p_{n+1}^{(m)})=p_0^{(m)}$.
So assume that $j\geq 1$.
If $u\notin X_{n+1,j}^{(m)}$, then by Lemma \ref{hutau}, 
$h_u(p_{n+1}^{(m)})=p_j^{(m)}$. 

So we assume that $u=u_1u_2\cdots u_{n+1}\in X_{n+1,j}^{(m)}$ and let
$s=u_1u_2\cdots u_{n}$.
Then we have that $s\in \varwordsj mn{j-1}
\subseteq \varwordsj {m+1}n{j-1}$ and hence 
$h_u(p_n^{(m)})=h_s(p_n^{(m)})=p_{j-1}^{(m)}$ and 
$h_u(p_n^{(m+1)})=h_s(p_n^{(m+1)})=p_{j-1}^{(m+1)}$.

Combined with  Lemma \ref{hutau}, we have that 
$$h_u\big(\tau_{n+1}^{(m+1)}(p_n^{(m+1)})\big)
=\tau_j^{(m+1)}\big(h_u(p_n^{(m+1)})\big)=\tau_j^{(m+1)}\big(p_{j-1}^{(m+1)}\big)\,.$$
So as 
$p_{n+1}^{(m)}\in G_{n+1}^{(m)}=
p_{n}^{(m)}\tau^{(m+1)}_{n+1}(p^{(m+1)}_{n})\beta S_{n+1}^{(m)}\cap 
\beta S_{n+1}^{(m)}\tau_{n+1}^{(m+1)}(p^{(m+1)}_{n})p_{n}^{(m)}$ we deduce that 
$h_u(p_{n+1}^{(m)})\in 
p_{j-1}^{(m)}\tau^{(m+1)}_j(p^{(m+1)}_{j-1})\beta S_{j}^{(m)}\cap 
\beta S_{j}^{(m)}\tau_{j}^{(m+1)}(p^{(m+1)}_{j-1})p_{j-1}^{(m)}\,.$
If $j\geq 2$, this says that $h_u(p_{n+1}^{(m)})$ is an idempotent in $G_j^{(m)}$
and $p_j^{(m)}$ is the identity of $G_j^{(m)}$ so 
$h_u(p_{n+1}^{(m)})=p_j^{(m)}$ as required.

Finally, assume that $j=1$. Then 
$$\begin{array}{rl}
h_u(p_{n+1}^{(m)})&\in p_0^{(m)}\tau_1^{(m+1)}(p_0^{(m+1)})\beta S_1^{(m)}\cap
\beta S_1^{(m)}\tau_1^{(m+1)}(p_0^{(m+1)})p_0^{(m)}\\
&=p_0^{(m)}p_1^{(m)}\beta S_1^{(m)}\cap
\beta S_1^{(m)}p_1^{(m)}p_0^{(m)}\\
&=p_1^{(m)}\beta S_1^{(m)}\cap\beta S_1^{(m)}p_1^{(m)}\,.\end{array}$$
Since $p_1^{(m)}$ is minimal in $\beta S_1^{(m)}$,
$p_1^{(m)}\beta S_1^{(m)}\cap\beta S_1^{(m)}p_1^{(m)}$ is a group with identity
$p_1^{(m)}$, so $h_u(p_{n+1}^{(m)})=p_1^{(m)}$.\end{proof}

Combining Lemma~\ref{height2} and Theorem~\ref{hyperextend} we obtain:

\begin{corollary}\label{hyperinf} For each $p\in E\big(K(\beta S_0)\big)$ 
there is a reductive array $\langle p_i^{(m)}\rangle_{i<\omega}^{m<\omega}$
of height $\omega$ such that $p_0^{(0)}=p$. Morever $p_1^{(0)}$ may be taken to be any minimal idempotent of $\beta S_1$
such that $p_1^{(0)}\leq p$.
  \end{corollary}

\begin{corollary}\label{infGR} 
Let $p$ be a minimal idempotent in $\beta S_0$.  
There is a sequence $\langle p_n\rangle_{n=0}^{\infty}$ such that \begin{itemize} 
\item[(1)] $p_0=p\,$;
\item[(2)] for each $n\in\ben$, $p_n$ is a minimal idempotent of
$\beta \widetilde S_n$;
\item[(3)] for each $n\in\ben$, $p_n\leq p_{n-1}$;
\item[(4)] for each $n\in\ben$, each $j\in\ohat{n-1}$, and each $u\in\varwords nj$,
$h_u(p_n)=p_j$.
\end{itemize}

\noindent Further, $p_1$ can be any minimal idempotent of $\beta S_1$
such that $p_1\leq p_0$. \end{corollary}

\begin{proof} Let $\langle p_i^{(m)}\rangle_{i<\omega}^{m<\omega}$ be
as guaranteed by Corollary \ref{hyperinf} and for each $i<\omega$ let
$p_i=p_i^{(0)}$. By Lemma \ref{varik} each $p_n\in E\big(K(\beta\widetilde S_n)\big)$.
\end{proof}

For several stronger combinatorial consequences of Corollary \ref{infGR}, see 
Sections 3 and 4 of \cite{CHS}.

To derive the following extension of Theorem \ref{GRth},
we need to restrict to a finite alphabet.

\begin{corollary}\label{corGR} Assume that $\A$ is finite and for
each $m<\omega$, let $\varphi_m$ be a finite coloring of $\widetilde S_m$. For
each $m<\omega$, there exists a central subset $C_m$ of $\widetilde S_m$ such that
\begin{itemize}
\item[(1)] $\varphi_m$ is constant on $C_m$ and
\item[(2)] whenever $n\in\ben$, the set of all $w\in \widetilde S_n$  
such that  for each $m<n$,\hfill\break $\{w(u):u\in\varwords nm\}\subseteq C_m$ is 
central in $\widetilde S_n$
\end{itemize}
\end{corollary}

\begin{proof} Pick $\langle p_m\rangle_{m<\omega}$ as guaranteed 
by Corollary \ref{infGR} and for each $m<\omega$ pick $C_m\in p_m$ with
$C_m\subseteq \widetilde S_m$ such that
$\varphi_m$ is constant on $C_m$. Let $n\in\ben$ be given.
For each $m<n$ and each $u\in\varwords nm$,
$h_u(p_n)=p_m$. Let $D=\bigcap_{m<n}\bigcap(h_u^{-1}[C_m]:u\in\varwords nm)$.
Then $D\in p_n$. 
\end{proof}

\bibliographystyle{plain}

\end{document}